\documentclass[12pt]{article}
\usepackage{amsmath}
\usepackage{fullpage}
\usepackage{amsfonts}
\usepackage{amssymb}
\usepackage{graphicx}
\usepackage{color}
\usepackage[hidelinks]{hyperref}
\usepackage{changes}
\usepackage{subcaption}
\usepackage{enumerate}

\newtheorem{theorem}{Theorem}[section]

\newtheorem{definition}[theorem]{Definition}

\newtheorem{lemma}[theorem]{Lemma}
\newtheorem{notation}[theorem]{Notation}

\newtheorem{proposition}[theorem]{Proposition}
\newtheorem{remark}[theorem]{Remark}

\newenvironment{proof}[1][Proof]{\textbf{#1.} }{\ \rule{0.5em}{0.5em}}

\makeatletter\@addtoreset{equation}{section}\makeatother

\newlength{\depthofsumsign}
\newcommand{\nsum}[1][1.4]{
	\mathop{%
		\raisebox
		{-#1\depthofsumsign+1\depthofsumsign}
		{\scalebox
			{#1}
			{$\displaystyle\sum$}%
		}
	}
}

\begin{document}

\title{On the Inf-Sup Stability of Crouzeix-Raviart Stokes Elements in 3D}
\author{Stefan Sauter\thanks{(stas@math.uzh.ch), Institut f\"{u}r Mathematik,
Universit\"{a}t Z\"{u}rich, Winterthurerstr 190, CH-8057 Z\"{u}rich,
Switzerland}
\and C\'{e}line Torres \thanks{(cetorres@umd.edu), University of Maryland,
Department of Mathematics, 4176 Campus Drive - William E. Kirwan Hall, College
Park MD 20742, USA}}
\maketitle

\begin{abstract}
We consider non-conforming discretizations of the stationary
Stokes equation in three spatial dimensions by Crouzeix-Raviart type elements.
The original definition in the seminal paper by M. Crouzeix and P.-A. Raviart
in 1973 is implicit and also contains substantial freedom for a concrete choice.

In this paper, we introduce \textit{basic} Crouzeix-Raviart basis
functions in 3D in analogy to the 2D case in a fully explicit way. We prove 
that this basic Crouzeix-Raviart element for the Stokes equation is
inf-sup stable for polynomial degree $k=2$ (quadratic velocity approximation).
We identify spurious pressure modes for the conforming $\left(
k,k-1\right)  $ 3D Stokes element and show that these are eliminated by
using the basic Crouzeix-Raviart space.

\end{abstract}



\section{Introduction}

In this paper, we investigate the discretization of the stationary Stokes
problem by Crouzeix-Raviart elements in three dimensions. They were introduced
in the seminal paper \cite{CrouzeixRaviart} in 1973 by Crouzeix and Raviart
with the goal to obtain a stable discretization of the Stokes equation with
relatively few unknowns. They can be considered as an non-conforming
enrichment of conforming finite elements of polynomial degree $k$ for the
velocity and discontinuous pressure of degree $k-1$. It 
is well known that
the conforming pair of finite elements can be unstable; for two dimensions the
proof of the inf-sup stability of Crouzeix-Raviart discretizations of general
order $k$ has been developed over the last 50 years, the inf-sup stability for
$k=1$ has been proved in \cite{CrouzeixRaviart} and only recently the last
open case $k=3$, has been proved in \cite{CCSS_CR_2}. We mention the papers
\cite{Fortin_Soulie}, \cite{ScottVogelius}, \cite{Crouzeix_Falk},
\cite{Baran_Stoyan}, \cite{GuzmanScott2019}, \cite{CCSS_CR_1} which contain
essential milestones for the proof of the inf-sup stability for 
general polynomial degree.
	There is a vast of literature on various further aspects of 
	Crouzeix-Raviart elements; we omit to present a comprehensive review here
	but refer to the excellent overview article \cite{Brenner_Crouzeix} 
	instead.

In contrast to the analysis in 2D, the development of 
Crouzeix-Raviart elements for the Stokes
equation in 3D is still in its infancy. The original definition in
\cite{CrouzeixRaviart} is implicit: a finite element space is a
Crouzeix-Raviart space if it satisfies certain jump conditions across the
element interfaces. This leaves significant freedom for a concrete definition.
In particular, the important question \textquotedblleft What is a
\textit{minimal} Crouzeix-Raviart space?\textquotedblright{}  i.e., 
\textquotedblleft What is a Crouzeix-Raviart space with minimal dimension 
so that the inf-sup condition is satisfied?\textquotedblright{}, is completely
open. For practical purpose, it is an even
stronger obstruction that the basis functions for such a minimal
Crouzeix-Raviart space are unknown. For \textit{quadratic} Crouzeix-Raviart
elements, an explicit basis has been introduced in \cite{Fortin_d3}. In
\cite{CDS}, a spanning set of functions is presented for a \textit{maximal}
Crouzeix-Raviart space of any order $k\in\mathbb{N}$ which allows for a
\textit{local basis}. However, the question of linear independence is subtle,
in particular, the definition of a basis for a \textit{minimal}
Crouzeix-Raviart space.

In this paper, we make a step in these directions. After formulating the
Stokes problem on the continuous level and introducing non-conforming finite
element discretizations in Section \ref{SecFEM}, we define in Section
\ref{SecCanCR} the \textit{basic} Crouzeix-Raviart space in three
dimensions for any polynomial degree $k\geq1$. We call them \textquotedblleft
basic\textquotedblright\ because they are in full analogy as in the 2D
case:
for odd polynomial degree $k$ there is one and only one scalar non-conforming
Crouzeix-Raviart function per inner facet (the analogue of a triangle edge in
2D) supported on the two adjacent simplices. As in the 2D case, this function
can be expressed by a certain orthogonal polynomial composed with barycentric
coordinates; for even polynomial degree $k$ there is one and only one scalar
Crouzeix-Raviart function per tetrahedron (in analogy to one Crouzeix-Raviart
function per triangle in 2D) supported on this tetrahedron. As 
in the 2D case,
these functions can be expressed by an orthogonal polynomial 
composed with
barycentric coordinates. The basic Crouzeix-Raviart spaces
can be defined conceptually in the same manner also in higher dimensions. We
have postponed the technical derivation of a fully explicit representation to
Appendix \ref{AppHigherD} due to the lack of practical applications in spatial
dimension larger than three.
For the Stokes problem and the corresponding basic 
Crouzeix-Raviart elements, we prove the following results.

a) In Section \ref{InfSup2}, we show that the basic Crouzeix-Raviart space
for $k=2$ on simplicial finite element meshes in 3D is inf-sup stable.

b) In Section \ref{SecCritConf}, we identify \textit{critical 
pressures} for
the conforming $\left(  k,k-1\right)  $ discretization of the Stokes problem.
They are related to the presence of \textit{critical edges }in the mesh, see
Def. \ref{CritEdge}. As a consequence, this conforming discretization is not 
inf-sup stable if the mesh contains critical edges.

c) In Section \ref{SecCRstab}, we show that these pressures are eliminated in
the basic Crouzeix-Raviart space. Hence, an inf-sup stable
Crouzeix-Raviart discretization should contain the basic
Crouzeix-Raviart space as a subspace while the question remains open, whether
the basic Crouzeix-Raviart Stokes element (cf. Def. 
\ref{DefCanCR}) is sufficient for an inf-sup stable discretization.

\section{Setting}
\label{SecFEM}

\subsection{The continuous Stokes problem\label{SecContStokes}}

Let $\Omega\subset\mathbb{R}^{3}$ denote a bounded polyhedral domain with
boundary $\partial\Omega$. We consider the Stokes equation%
\[%
\begin{array}
[c]{llll}%
-\Delta\mathbf{u} & -\nabla p & =\mathbf{f} & \text{in }\Omega,\\
\operatorname*{div}\mathbf{u} &  & =0 & \text{in }\Omega,
\end{array}
\]
with Dirichlet boundary conditions for the velocity and a normalization
condition for the pressure:
\[
\mathbf{u}=\mathbf{0}\quad\text{on }\partial\Omega\quad\text{and\quad}%
\int_{\Omega}p=0.
\]
To state the classical existence and uniqueness result, we 
introduce 
the relevant function spaces and formulate this equation in a variational 
form. Throughout the paper we restrict to vector spaces over the field of
real numbers.

For $s\geq0$, $1\leq p\leq\infty$, $W^{s,p}\left(  \Omega\right)  $ denotes the
classical Sobolev space of functions with norm $\left\Vert \cdot\right\Vert
_{W^{s,p}\left(  \Omega\right)  }$. As usual we write $L^{p}\left(
\Omega\right)  $ instead of $W^{0,p}\left(  \Omega\right)  $ and $H^{s}\left(
\Omega\right)  $ for $W^{s,2}\left(  \Omega\right)  $. For $s\geq0$, we denote
by $H_{0}^{s}\left(  \Omega\right)  $ the closure of the space of infinitely
smooth functions with compact support in $\Omega$ with respect to the
$H^{s}\left(  \Omega\right)  $ norm. Its dual space is denoted by
$H^{-s}\left(  \Omega\right)  $.

The scalar product and norm in $L^{2}\left(  \Omega\right)  $ 
are denoted by
\[%
\begin{array}
[c]{llll}%
\left(  u,v\right)  _{L^{2}\left(  \Omega\right)  }:=\int_{\Omega}uv &
\text{and} & \left\Vert u\right\Vert _{L^{2}\left(  \Omega\right)  }:=\left(
u,u\right)  _{L^{2}\left(  \Omega\right)  }^{1/2} & \text{in }L^{2}\left(
\Omega\right)  .
\end{array}
\]
Vector-valued and $3\times3$ tensor-valued analogues of the function spaces
are denoted by bold and blackboard bold letters, e.g., $\mathbf{H}^{s}\left(
\Omega\right)  =\left(  H^{s}\left(  \Omega\right)  \right)  ^{3}$ and
$\mathbb{H}^{s}=\left(  H^{s}\left(  \Omega\right)  \right)  ^{3\times3}$ and
analogously for other quantities.

The $\mathbf{L}^{2}\left(  \Omega\right)  $ scalar product and norm for vector
valued functions are given by%
\[
\left(  \mathbf{u},\mathbf{v}\right)  _{\mathbf{L}^{2}\left(  \Omega\right)
}:=\int_{\Omega}\left\langle \mathbf{u},\mathbf{v}\right\rangle \quad
\text{and\quad}\left\Vert \mathbf{u}\right\Vert _{\mathbf{L}^{2}\left(
\Omega\right)  }:=\left(  \mathbf{u},\mathbf{u}\right)  _{\mathbf{L}%
^{2}\left(  \Omega\right)  }^{1/2},
\]
where $\left\langle \mathbf{u},\mathbf{v}\right\rangle $ denotes the Euclidean
scalar product in $\mathbb{R}^{3}$. In a similar fashion, we define for
$\mathbf{G},\mathbf{H}\in\mathbb{L}^{2}\left(  \Omega\right)  $ the scalar
product and norm by%
\[
\left(  \mathbf{G},\mathbf{H}\right)  _{\mathbb{L}^{2}\left(  \Omega\right)
}:=\int_{\Omega}\left\langle \mathbf{G},\mathbf{H}\right\rangle \quad
\text{and\quad}\left\Vert \mathbf{G}\right\Vert _{\mathbb{L}^{2}\left(
\Omega\right)  }:=\left(  \mathbf{G},\mathbf{G}\right)  _{\mathbb{L}%
^{2}\left(  \Omega\right)  }^{1/2},
\]
where $\left\langle \mathbf{G},\mathbf{H}\right\rangle =\sum_{i,j=1}%
^{3}G_{i,j}H_{i,j}$. Finally, let $L_{0}^{2}\left(  \Omega\right)  :=\left\{
u\in L^{2}\left(  \Omega\right)~|~\int_{\Omega}u=0\right\}  
$.

We introduce the bilinear forms $a:\mathbf{H}^{1}_0\left(  
\Omega\right)
\times\mathbf{H}^{1}_0\left(  \Omega\right)  \rightarrow\mathbb{R}$ and
$b:L_{0}^{2}\left(  \Omega\right)  \times\mathbf{H}_{0}^{1}\left(
\Omega\right)  \rightarrow\mathbb{R}$ by
\begin{equation}
a\left(  \mathbf{u},\mathbf{v}\right)  :=\left(  \nabla\mathbf{u}%
,\nabla\mathbf{v}\right)  _{\mathbb{L}^{2}\left(  \Omega\right)  },\qquad
b\left(  p,\mathbf{v}\right)  :=\left(  p,\operatorname*{div}\mathbf{v}%
\right)  _{L^{2}\left(  \Omega\right)  }, \label{defabili}%
\end{equation}
where $\nabla\mathbf{u}$ and $\nabla\mathbf{v}$ denote the derivatives of
$\mathbf{u}$ and $\mathbf{v}$. The variational form of the Stokes problem is
given by: Given $\mathbf{f}\in\mathbf{H}^{-1}\left(  \Omega\right)  ,$%
\begin{equation}
\text{find }\left(  \mathbf{u},p\right)  \in\mathbf{H}_{0}^{1}\left(
\Omega\right)  \times L_{0}^{2}\left(  \Omega\right)  \;\text{s.t.\ }\left\{
\begin{array}
[c]{llll}%
a\left(  \mathbf{u},\mathbf{v}\right)  +&b\left(  
p,\mathbf{v}\right) & =\mathbf{f}(\mathbf{v}) & \forall
\mathbf{v}\in\mathbf{H}_{0}^{1}\left(  \Omega\right)  ,\\
b\left(  q,\mathbf{u}\right)  && =0 & \forall q\in L_{0}^{2}\left(
\Omega\right)  .
\end{array}
\right.  \label{varproblemstokes}%
\end{equation}
It is well-known, that (\ref{varproblemstokes})
is well posed (see, e.g., \cite{Girault86}). Since we consider non-conforming 
discretizations, we restrict
the space $\mathbf{H}^{-1}\left(  \Omega\right)  $ for the right-hand side to
a smaller space and assume for simplicity that $\mathbf{f}\in\mathbf{L}%
^{2}\left(  \Omega\right)  $, i.e. we write $\mathbf{f}(\mathbf{v})$ as
$\left(  \mathbf{f},\mathbf{v}\right)  _{\mathbf{L}^{2}\left(  \Omega\right)
}$.

\subsection{Discretization}

Given two finite-dimensional approximation spaces $\mathbf{V}_{h}$ 
with an appropriate norm \(\|\cdot\|_{\mathbf{V}_{h}}\) for the
velocity and $M_{h}$ for the pressure, a finite element approximation of
(\ref{varproblemstokes}) then reads: For given $\mathbf{f}\in\mathbf{L}%
^{2}\left(  \Omega\right)  ,$%
\begin{equation}
\text{find }\left(  \mathbf{u}_{h},p_{h}\right)  \in\mathbf{V}_{h}\times
M_{h}\;\text{s.t.\ }\left\{
\begin{array}
[c]{llll}%
a_{h}\left(  \mathbf{u}_{h},\mathbf{v}\right)  +&b_{h}\left(  p_{h}%
,\mathbf{v}\right)  & =\left(  \mathbf{f},\mathbf{v}\right)  _{\mathbf{L}%
^{2}\left(  \Omega\right)  } & \forall\mathbf{v}\in\mathbf{V}_{h},\\
b_{h}\left(  q,\mathbf{u}_{h}\right) & & =0 & \forall q\in M_{h}.
\end{array}
\right.  \label{Stokesdiscweak}%
\end{equation}
Here, $a_{h}\left(  \cdot,\cdot\right)  :\mathbf{V}_{h}\times\mathbf{V}%
_{h}\rightarrow\mathbb{R}$ is a discrete version of the bilinear form
$a\left(  \cdot,\cdot\right)  $ in (\ref{defabili}) which is defined on the
discrete space $\mathbf{V}_{h}$ and $b_{h}\left(  \cdot,\cdot\right)
:M_{h}\times\mathbf{V}_{h}\rightarrow\mathbb{R}$ is a discrete version of
$b\left(  \cdot,\cdot\right)  $ in (\ref{defabili}). For the 
choice of Crouzeix-Raviart elements, we will give the concrete definition of 
\(\|\cdot\|_{\mathbf{V}_{h}}\) and the bilinear forms \(a_{h}\), \(b_{h}\)  
in Definition \ref{DefCanCR}. It is well known that if $a_{h}\left(  
\cdot,\cdot\right)  $,
$b_{h}\left(  \cdot,\cdot\right)  $ are continuous, $a_{h}\left(  \cdot
,\cdot\right)  $ is symmetric and $\mathbf{V}_{h}$--coercive, and the spaces
$\mathbf{V}_{h}$ and $M_{h}$ satisfy the inequality
\[
\inf_{p\in M_{h}\backslash\left\{  0\right\}  }\sup_{\mathbf{v}\in
\mathbf{V}_{h}\backslash\left\{  \mathbf{0}\right\}  }\frac{b_{h}\left(
p,\mathbf{v}\right)  }{\left\Vert \mathbf{v}\right\Vert 
_{\mathbf{V}_{h}}\left\Vert p\right\Vert _{L^{2}\left(
\Omega\right)  }}\geq\gamma>0,
\]
then the discrete problem is well-posed. In this case, we call the pair
$\left(  \mathbf{V}_{h},M_{h}\right)  $ \textit{inf-sup stable}.

\section{Basic Crouzeix-Raviart finite elements in 3D}
\label{SecCanCR}
In the following, we define Crouzeix-Raviart spaces in 3D for the velocity
discretization in (\ref{varproblemstokes}). Let $\mathcal{T}$ be a conforming
finite element mesh for $\Omega$ consisting of closed tetrahedra
$K\in\mathcal{T}$. We denote by $\hat{K}$ the reference tetrahedron with
vertices
\begin{equation}%
\begin{array}
[c]{llll}%
\mathbf{\hat{z}}_{1}:=\mathbf{0}, & \mathbf{\hat{z}}_{2}:=\left(
1,0,0\right)  ^{T}, & \mathbf{\hat{z}}_{3}:=\left(  0,1,0\right)  ^{T}, &
\mathbf{\hat{z}}_{4}:=\left(  0,0,1\right)  .
\end{array}
\label{Defzhat}%
\end{equation}
Moreover, let $\mathcal{F}$ ($\mathcal{E}$, $\mathcal{V}$, resp.) be the set
of all two-dimensional facets (one-dimensional edges, vertices, resp.) in the
mesh and let%
\[%
\begin{array}
[c]{ll}%
\mathcal{F}_{\partial\Omega}:=\left\{  F\in\mathcal{F}\mid F\subset
\partial\Omega\right\}  , & \mathcal{F}_{\Omega}:=\mathcal{F}\backslash
\mathcal{F}_{\partial\Omega},\\
\mathcal{E}_{\partial\Omega}:=\left\{  E\in\mathcal{E}\mid E\subset
\partial\Omega\right\}  , & \mathcal{E}_{\Omega}:=\mathcal{E}\backslash
\mathcal{E}_{\partial\Omega},\\
\mathcal{V}_{\partial\Omega}:=\left\{  \mathbf{z}\in\mathcal{V}\mid
\mathbf{z}\in\partial\Omega\right\}  , & \mathcal{V}_{\Omega}:=\mathcal{V}%
\backslash\mathcal{V}_{\partial\Omega}.
\end{array}
\]
For $F\in\mathcal{F}$, $E\in\mathcal{E}$, $\mathbf{z}\in\mathcal{V}$,
we define facet, edge, nodal patches by
\begin{equation}%
\begin{array}
[c]{ll}%
\mathcal{T}_{F}:=\left\{  K\in\mathcal{T}:F\subset K\right\}  , & \omega
_{F}:=\bigcup_{K\in\mathcal{T}_{F}}K,\\
\mathcal{T}_{E}:=\left\{  K\in\mathcal{T}:E\subset K\right\}  , & \omega
_{E}:=\bigcup_{K\in\mathcal{T}_{E}}K,\\
\mathcal{T}_{\mathbf{z}}:=\left\{  K\in\mathcal{T}:\mathbf{z}\in K\right\}
, & \omega_{\mathbf{z}}:=\bigcup_{K\in\mathcal{T}_{\mathbf{z}}}K.
\end{array}
\label{DefPatches}%
\end{equation}
For a subset of tetrahedra $\mathcal{M}\subset\mathcal{T}$, we define the
patch
\[
\operatorname*{dom}\mathcal{M}:=\operatorname*{int}\left(  \bigcup
_{K\in\mathcal{M}}K\right)  ,
\]
where $\operatorname*{int}\left(  D\right)  $ denotes the interior of a set
$D\subset\mathbb{R}^{3}$. For a measurable subset $D\subset\mathbb{R}^{d}$, we
denote by $\left\vert D\right\vert _{d}$ the $d$-dimensional volume of $D$ and
skip the index $d$ if the dimension is clear from the context, e.g.,
$\left\vert K\right\vert $ denotes the three-dimensional volume of a simplex
$K\in\mathcal{T}$ and $\left\vert F\right\vert $ the two-dimensional area of a
facet $F\in\mathcal{F}$.

For a conforming simplicial mesh $\mathcal{T}$ of the domain $\Omega$, let
\[
H^{1}\left(  \mathcal{T}\right)  :=\left\{  u\in L^{2}\left(  
\Omega\right)~
	\Big | ~\forall K\in\mathcal{T}:~\left.  u\right\vert 
	_{\overset{\circ}{K}}\in
	H^{1}\left(  \overset{\circ}{K}\right)  \right\}.
\]

Let $\mathbb{N}:=\left\{  1,2,3,\ldots\right\}  $ and $\mathbb{N}%
_{0}:=\mathbb{N}\cup\left\{  0\right\}  $. For $n\in\mathbb{N}_{0}$ and a
domain $D\subset\mathbb{R}^{3}$ we denote by $\mathbb{P}_{n}(D)$ the space of
polynomials of maximal degree $n$ on $D$ and set $\mathbb{P}_{-1}\left(
D\right)  :=\left\{  0\right\}  $. We introduce the following finite element
spaces. For $k\in\mathbb{N}$, let $S_{k,0}(\mathcal{T})$ denote the space of
globally continuous, piecewise polynomials of degree $\leq k$ with vanishing
trace on the boundary
\[
S_{k,0}(\mathcal{T}):=\left\{  v\in C^{0}(\operatorname*{dom}\mathcal{T}%
)~\big |~\left.  v\right\vert 
_{K}\in\mathbb{P}_{k}(K)\quad\forall 
K\in
\mathcal{T}\quad\wedge\quad v=0\text{ on }\partial\left(  \operatorname*{dom}%
\mathcal{T}\right)  \right\}  .
\]
Its vector valued version is $\mathbf{S}_{k,0}(\mathcal{T}):=(S_{k,0}%
(\mathcal{T}))^{3}$, which is a conforming subspace of 
$\mathbf{H}_{0}^{1}\left(
\Omega\right)  $. The space of discontinuous polynomials of maximal degree
$k-1$ is%
\[
\mathbb{P}_{k-1}\left(  \mathcal{T}\right)  :=\left\{  p\in L^{2}%
\left(\Omega\right) ~\Big |~\left.  p\right\vert 
_{\overset{\circ}{K}}\in\mathbb{P}%
_{k-1}(\overset{\circ}{K})\quad\forall K\in\mathcal{T}\right\}
\]
and the subspace $\mathbb{P}_{k-1,0}\left(  \mathcal{T}\right)  $ is given by%
\[
\mathbb{P}_{k-1,0}\left(  \mathcal{T}\right)  :=\left\{  p\in\mathbb{P}%
_{k-1}\left(  \mathcal{T}\right) ~\Big 
|~\int_{\operatorname*{dom}\mathcal{T}%
}p=0\right\}  .
\]

\begin{notation}
\label{NotBary}For a simplex $K$, its four vertices form the set
$\mathcal{V}\left(  K\right)  $ and its four facets form the set 
\(\mathcal{F}(K)\). For a triangular facet, its three vertices
form the set $\mathcal{V}\left(  F\right)  $, and for an edge $E$, the two
endpoints form the set $\mathcal{V}\left(  E\right)  $. Let $\delta
_{\mathbf{x},\mathbf{y}}$ denote Kronecker's delta. The \emph{barycentric
coordinate} for a vertex $\mathbf{z}\in\mathcal{V}\left(  K\right)  $ is
defined by the conditions%
\[
\lambda_{K,\mathbf{z}}\in\mathbb{P}_{1}\left(  K\right)  \quad\text{and}%
\quad\lambda_{K,\mathbf{z}}\left(  \mathbf{y}\right)  =\delta_{\mathbf{z}%
,\mathbf{y}}\text{ \quad for all vertices }\mathbf{y}\in\mathcal{V}\left(
K\right)  .
\]
For $%
\boldsymbol{\mu}%
=\left(  \mu_{\mathbf{y}}\right)  _{\mathbf{y}\in\mathcal{V}\left(  K\right)
}\in\mathbb{N}_{0}^{4}$, we set%
\[%
 \boldsymbol{\lambda}%
_{K}^{
\boldsymbol{\mu}
}:=
{\displaystyle\prod\limits_{\mathbf{y}\in\mathcal{V}\left(  K\right)  }}
\lambda_{K,\mathbf{y}}^{\mu_{\mathbf{y}}}.
\]
For a facet $F\subset\partial K$ and $%
\boldsymbol{\mu}%
=\left(  \mu_{\mathbf{y}}\right)  _{\mathbf{y}\in\mathcal{V}\left(  F\right)
}\in\mathbb{N}_{0}^{3}$, we set%
\[
 \boldsymbol{\lambda}%
_{K,F}^{%
\boldsymbol{\mu}%
}:=
{\displaystyle\prod\limits_{\mathbf{y}\in\mathcal{V}\left(  F\right)  }}
\lambda_{K,\mathbf{y}}^{\mu_{\mathbf{y}}}.
\]
Finally, for an edge $E\subset\partial K$ and $%
\boldsymbol{\mu}%
=\left(  \mu_{\mathbf{y}}\right)  _{\mathbf{y}\in\mathcal{V}\left(  E\right)
}\in\mathbb{N}_{0}^{2}$, we set
\[%
 \boldsymbol{\lambda}_{K,E}^{%
\boldsymbol{\mu}
_{E}}:=%
{\displaystyle\prod\limits_{\mathbf{y}\in\mathcal{V}\left(  E\right)  }}
\lambda_{K,\mathbf{y}}^{\mu_{\mathbf{y}}}.
\]
We also set $\mathbf{1}_{K}:=\left(  1\right)  _{\mathbf{y}\in\mathcal{V}%
\left(  K\right)  }$, $\mathbf{1}_{F}:=\left(  1\right)  _{\mathbf{y}%
\in\mathcal{V}\left(  F\right)  }$, $\mathbf{1}_{E}:=\left(  1\right)
_{\mathbf{y}\in\mathcal{V}\left(  E\right)  }$. For $\mathbf{y}\in
\mathcal{V}\left(  K\right)  $, we set $\mathbf{e}_{\mathbf{y}}^{K}:=\left(
\delta_{\mathbf{y},\mathbf{x}}\right)  _{\mathbf{x}\in\mathcal{V}\left(
K\right)  }$, for a facet $F$ and $\mathbf{y}\in\mathcal{V}\left(  F\right)
$, we set $\mathbf{e}_{\mathbf{y}}^{F}:=\left(  \delta_{\mathbf{y},\mathbf{x}%
}\right)  _{\mathbf{x}\in\mathcal{V}\left(  F\right)  }$, and for an edge $E$
and $\mathbf{y}\in\mathcal{V}\left(  E\right)  $ let $\mathbf{e}_{\mathbf{y}%
}^{E}:=\left(  \delta_{\mathbf{y},\mathbf{x}}\right)  _{\mathbf{x}%
\in\mathcal{V}\left(  E\right)  }$.
\end{notation}

Next, we define the non-conforming Crouzeix-Raviart space. For 
a 
function $v\in
H^{1}\left(  \mathcal{T}\right)  $, we denote by $\left[  v\right]  _{F}$ the
jump of $v\in\mathbb{P}_{k}(\mathcal{T})$ across the facet $F\in\mathcal{F}$
and $\mathbb{P}_{k-1}(F)$ is the space of polynomials of maximal degree $k-1$
with respect to the local variables in $F$. For $k\geq1$ and any
$F\in\mathcal{F}$, let%
\[
\mathbb{P}_{k,k-1}^{\perp}\left(  F\right)  :=\left\{  q\in\mathbb{P}%
_{k}\left(  F\right)  \mid\left(  q,r\right)  _{L^{2}\left(  F\right)
}=0\quad\forall r\in\mathbb{P}_{k-1}\left(  F\right)  \right\}  .
\]
The scalar version of the Crouzeix-Raviart space of order $k$ is defined
implicitly by%
\begin{equation}
\operatorname*{CR}\nolimits_{k,0}^{\max}\left(  \mathcal{T}\right)  :=\left\{
v\in\mathbb{P}_{k}\left(  \mathcal{T}\right)~\Big |~\left(
\begin{array}
[c]{cl}
& \forall F\in\mathcal{F}_{\Omega}\quad\left[  v\right]  _{F}\in
\mathbb{P}_{k,k-1}^{\perp}\left(  F\right) \\
\text{and} & \forall F\in\mathcal{F}_{\partial\Omega}\quad v\in\mathbb{P}%
_{k,k-1}^{\perp}\left(  F\right)
\end{array}
\right)  \right\}  . \label{DefCRmax}%
\end{equation}
Its vector version is denoted by $\mathbf{CR}_{k,0}^{\max}(\mathcal{T}%
):=(\operatorname*{CR}_{k,0}^{\max}(\mathcal{T}))^{3}$. We also define%
\begin{equation}
S_{k,0}^{\prime}\left(  \mathcal{T}\right)  :=\left\{  v\in S_{k,0}\left(
\mathcal{T}\right)  \mid v(\mathbf{z})=0\quad\forall\mathbf{z}\in
\mathcal{V}\right\}  \label{DefSkprime}%
\end{equation}
as the subspace of $S_{k,0}\left(  \mathcal{T}\right)  $ consisting of
functions which vanishes at the vertices of the mesh.

\begin{remark}
\label{Remmaxcan}In two spatial dimensions, local basis functions for
Crouzeix-Raviart spaces have been defined in \cite{BaranCVD},
\cite{Ainsworth_Rankin}, \cite[for $p=4,6$]{ChaLeeLee}, \cite{ccss_2012},
\cite{Baran_Stoyan}, \cite{CCSS_CR_1}. It turns out that the non-conforming
Crouzeix-Raviart basis functions of \emph{odd }polynomial degree $k$ are
associated to the inner triangle edges while for \emph{even} polynomial degree
they are associated to the triangles in the mesh.

The situation is much more complicated in 3D. In \cite{CDS} local shape
functions are introduced which span the Crouzeix-Raviart space
$\operatorname*{CR}_{k,0}^{\max}\left(  \mathcal{T}\right)  $ and it was shown
that per inner facet $\mathcal{F}$ there exist $\left\lfloor \frac{k+2}%
{3}\right\rfloor $ linearly independent, non-conforming Crouzeix-Raviart
functions and, in addition, per simplex, there exist $\left\lfloor \frac{k}%
{2}\right\rfloor -\left\lfloor \frac{k-1}{2}\right\rfloor $ linearly
independent, non-conforming Crouzeix-Raviart functions. We say that any space
$V$ with $S_{k,0}\left(  \mathcal{T}\right)  \subsetneqq V\subset
\operatorname*{CR}_{k,0}^{\max}\left(  \mathcal{T}\right)  $ is 
\emph{a Crouzeix-Raviart space}.

A natural question is whether there is an analogous choice of $V$ as in two
dimensions: One Crouzeix-Raviart function per facet for odd polynomial degree
and one Crouzeix-Raviart function per tetrahedron for even polynomial degree.
In addition, these Crouzeix-Raviart functions should have a 
\textquotedblleft
similarly simple\textquotedblright\ representation as those in 2D, see
\cite[Def. 3.2]{CCSS_CR_1}. We call the space $S_{k,0}\left(  \mathcal{T}%
\right)  $, enriched by those functions, the \emph{basic Crouzeix-Raviart
space}. Since this space is smaller than $\operatorname*{CR}_{k,0}^{\max
}\left(  \mathcal{T}\right)  $ we have used the superscript \textquotedblleft%
$\max$\textquotedblright\ in (\ref{DefCRmax}).
\end{remark}

Next, we define the basic Crouzeix-Raviart functions on simplicial
meshes in 3D. Let $\alpha,\beta>-1$ and $n\in\mathbb{N}_{0}$. The \emph{Jacobi
polynomial} $P_{n}^{\left(  \alpha,\beta\right)  }$ is a polynomial of degree
$n$ such that
\[
\int_{-1}^{1}P_{n}^{\left(  \alpha,\beta\right)  }\left(  x\right)  \,q\left(
x\right)  \left(  1-x\right)  ^{\alpha}\left(  1+x\right)  ^{\beta}\,dx=0
\]
for all polynomials $q$ of degree less than $n$, and (cf. \cite[Table
18.6.1]{NIST:DLMF})%
\begin{equation}
P_{n}^{\left(  \alpha,\beta\right)  }\left(  1\right)  =\frac{\left(
\alpha+1\right)  _{n}}{n!},\qquad P_{n}^{\left(  \alpha,\beta\right)  }\left(
-1\right)  =\left(  -1\right)  ^{n}\frac{\left(  \beta+1\right)  _{n}}{n!}.
\label{Pnormalization}%
\end{equation}
Here, the \emph{shifted factorial} is defined by $\left(  a\right)
_{n}:=a\left(  a+1\right)  \ldots\left(  a+n-1\right)  $ for $n>0$ and
$\left(  a\right)  _{0}:=1$. Note that $P_{k}^{\left(  0,0\right)  }$ are the
Legendre polynomials (see \cite[18.7.9]{NIST:DLMF}) and we set $L_{k}%
:=P_{k}^{\left(  0,0\right)  }$. For later use, we state an orthogonality
relation on a tetrahedron for polynomials which are related to $P_{k}^{\left(
0,3\right)  }$.

\begin{lemma}
\label{LemOrthoMain}For a tetrahedron $K$ with barycentric 
coordinates
$\lambda_{K,\mathbf{y}}$, $\mathbf{y}\in\mathcal{V}\left(  K\right)  $, the
polynomial $P_{k}^{\left(  0,3\right)  }(1-2\lambda_{K,\mathbf{y}})$ is
orthogonal to $\mathbb{P}_{k-1}(K)$ with respect to the weight functions%
\[
\lambda_{K,\mathbf{z}},\quad\mathbf{z}\in\mathcal{V}\left(  K\right)
\backslash\left\{  \mathbf{y}\right\}  .
\]

\end{lemma}

The assertion is a particular case of \cite[Prop. 2.3.8]{dunkl2014orthogonal}.
Since the proof for our concrete case is very simple we give it here for completeness.%

\begin{proof}
Let $\mathbf{y},\mathbf{z}\in\mathcal{V}\left(  K\right)  $ with
$\mathbf{y}\neq\mathbf{z}$ and let $\chi_{K}:\hat{K}\rightarrow K$ be an
affine pullback such that $\lambda_{K,\mathbf{y}}\circ\chi_{K}\left(
\mathbf{x}\right)  =x_{1}$ and $\lambda_{K,\mathbf{z}}\circ\chi_{K}\left(
\mathbf{x}\right)  =x_{2}$. For $\boldsymbol{\alpha}
=\left(  \alpha_{j}\right)  _{j=1}^{3}\in\mathbb{N}_{0}^{3}$ with $\left\vert
\boldsymbol{\alpha}
\right\vert \leq k-1$, let 
$\hat{q}^{\boldsymbol{\alpha}} (\boldsymbol{x})
:=x_{1}^{\alpha_{1}}x_{2}^{\alpha_{2}}%
	x_{3}^{\alpha_{3}}$ and note that $\mathbb{P}_{k-1}(K)$ is spanned by the
lifted versions 
$\hat{q}^{\boldsymbol{\alpha}}\circ\chi_{K}^{-1}$. Then,
{\allowdisplaybreaks
\begin{align*}
\int_{K}P_{k}^{\left(  0,3\right)  }  &  \left(  1-2\lambda_{K,\mathbf{y}%
}\right)  \lambda_{K,\mathbf{z}} ~ \hat{q}^{\boldsymbol{\alpha}}
\circ\chi_{K}^{-1} \\
&  =\frac{\left\vert K\right\vert }{\left\vert \hat{K}\right\vert }\int%
_{0}^{1}\int_{0}^{1-x_{1}}\int_{0}^{1-x_{1}-x_{2}}P_{k}^{\left(  0,3\right)
}\left(  1-2x_{1}\right)  x_{1}^{\alpha_{1}}x_{2}^{\alpha_{2}+1}x_{3}%
^{\alpha_{3}}dx_{3}dx_{2}dx_{1}\\
&  =\frac{1}{\alpha_{3}+1}\frac{\left\vert K\right\vert }{\left\vert \hat
{K}\right\vert }\int_{0}^{1}\int_{0}^{1-x_{1}}P_{k}^{\left(  0,3\right)
}\left(  1-2x_{1}\right)  x_{1}^{\alpha_{1}}x_{2}^{\alpha_{2}+1}\left(
1-x_{1}-x_{2}\right)  ^{\alpha_{3}+1}dx_{2}dx_{1}\\
&  =\frac{\left(  \alpha_{2}+1\right)  !\alpha_{3}
!}{\left(  \alpha_{2}+\alpha_{3}+3\right)  !}\frac{\left\vert K\right\vert
}{\left\vert \hat{K}\right\vert }\int_{0}^{1}P_{k}^{\left(  0,3\right)
}\left(  1-2x_{1}\right)  \left(  1-x_{1}\right)  ^{3}x_{1}^{\alpha_{1}%
}\left(  1-x_{1}\right)  ^{\alpha_{2}+\alpha_{3}}dx_{1}\\
&  \overset{1-2x_{1}=t}{=}\frac{\left(  \alpha_{2}+1\right)  !\alpha_{3}
		!}{16\left(  \alpha_{2}+\alpha_{3}+3\right)  !}%
\frac{\left\vert K\right\vert }{\left\vert \hat{K}\right\vert }\int_{-1}%
^{1}P_{k}^{\left(  0,3\right)  }\left(  t\right)  \left(  t+1\right)
^{3}\underbrace{\left(  \frac{1-t}{2}\right)  ^{\alpha_{1}}\left(  \frac
{1+t}{2}\right)  ^{\alpha_{2}+\alpha_{3}}}_{\in\mathbb{P}_{k-1}(\mathbb{R}%
)}dt\\
&  =0,
\end{align*}}
by the orthogonality properties of the Jacobi polynomials.%
\end{proof}

We introduce univariate polynomials $Q_{k}\in\mathbb{P}_{k}$ by%
\begin{equation}
Q_{k}:=\frac{1}{k+1}\left(  L_{k+1}-L_{k}\right)  ^{\prime}. \label{defQk2}%
\end{equation}

\begin{definition}
Let $\mathcal{T}$ be a conforming simplicial finite element mesh in 3D.

\begin{enumerate}
\item For even $k\geq2$, and any $K\in\mathcal{T}$, the simplex-oriented
\emph{Crouzeix-Raviart basis function} $B_{k}^{\operatorname*{CR},K}%
\in\mathbb{P}_{k}\left(  \mathcal{T}\right)  $ is given by
\begin{equation}
B_{k}^{\operatorname*{CR},K}:=\left\{
\begin{array}
[c]{ll}%
\left(  \sum_{\mathbf{z}\in\mathcal{V}\left(  K\right)  }Q_{k}\left(
1-2\lambda_{K,\mathbf{z}}\right)  \right)  -1 & \text{on }K\text{,}\\
0 & \text{otherwise.}%
\end{array}
\right.  \label{def:CRtetrahedron}%
\end{equation}

\item For odd $k\geq1$, and any $F\in\mathcal{F}_{\Omega}$, the facet-oriented
\emph{Crouzeix-Raviart basis function} $B_{k}^{\operatorname*{CR},F}%
\in\mathbb{P}_{k}\left(  \mathcal{T}\right)  $ is given by%
\begin{equation}
B_{k}^{\operatorname*{CR},F}:=\left\{
\begin{array}
[c]{ll}%
Q_{k}\left(  1-2\lambda_{K,\mathbf{z}}\right)  & \text{for }K\in
\mathcal{T}_{F},\\
0 & \text{otherwise,}%
\end{array}
\right.  \label{def:CRfacet}%
\end{equation}
where $\lambda_{K,\mathbf{z}}$ denotes the barycentric coordinate for the
vertex $\mathbf{z}\in\mathcal{V}\left(  K\right)  $ opposite to $F$%
.\footnote{Since $\lambda_{K,\mathbf{z}}=0$ on $F$, the (constant) values of
$B_{k}^{\operatorname*{CR},F}$ on $F$ from both sides coincide and are given
by $Q_{k}\left(  1\right)  =1$. This implies, that $B_{k}^{\operatorname*{CR}%
,F}$ is continuous across $F$.}
\end{enumerate}
\end{definition}

\begin{theorem}
\label{thm:CR-tetrahedron}\quad

\begin{enumerate}[(a)]
\item Let $K\in\mathcal{T}$ and $k\geq2$ be even. The function
$B_{k}^{\operatorname*{CR},K}$ in (\ref{def:CRtetrahedron}) is $L^{2}\left(
F\right)  $-orthogonal to $\mathbb{P}_{k-1}\left(  F\right)  $ on any facet
$F$ of $K$ and belongs to $\operatorname*{CR}\nolimits_{k,0}^{\max}\left(
\mathcal{T}\right)  $.

\item Let $F\in\mathcal{F}_{\Omega}$ and $k\geq1$ be odd. The function
$B_{k}^{\operatorname*{CR},F}$ in (\ref{def:CRfacet}) is $L^{2}\left(
F\right)  $-orthogonal to $\mathbb{P}_{p-1}\left(  F^{\prime}\right)  $ on any
outer facet $F^{\prime}\in\partial\omega_{F}$, continuous across $F$, and
belongs to $\operatorname*{CR}\nolimits_{k,0}^{\max}\left(  \mathcal{T}%
\right)  $.
\end{enumerate}
\end{theorem}

\begin{proof}
\textbf{(a): }Let $k\geq2$ be 
even. Let 
$F\in\mathcal{F}_{\Omega}$ with
$F\subset\partial K$ and let $\mathbf{z}\in\mathcal{V}\left(  K\right)  $
denote the vertex opposite to $F$. Recall that the Legendre polynomial
satisfies (cf. \cite[Combine 18.9.19, 18.7.9 and Table 18.6.1]{NIST:DLMF})%
\begin{equation}
L_{k}^{\prime}(\pm1)=\left(  \pm1\right)  ^{k-1}\binom{k+1}{2}.
\label{Lendpoints}%
\end{equation}
Since $\lambda_{K,\mathbf{z}}=0$ on $F$, the function $B_{k}%
^{\operatorname*{CR},K}$ can be expressed by%
\begin{align*}
\left.  {B_{k}^{\operatorname*{CR},K}}\right\vert _{F}  &  ={\left(
\sum_{\mathbf{y}\in\mathcal{V}\left(  F\right)  }\left.  Q_{k}\left(
1-2\lambda_{K,\mathbf{y}}\right)  \right\vert _{F}\right)  }+\frac
{L_{k+1}^{\prime}\left(  1\right)  -L_{k}^{\prime}\left(  1\right)  }{k+1}-1\\
&  =\sum_{\mathbf{y}\in\mathcal{V}\left(  F\right)  }Q_{k}\left(  {1-2}\left.
{\lambda_{K,\mathbf{y}}}\right\vert _{F}\right)  .
\end{align*}
For $\mathbf{y}\in\mathcal{V}\left(  F\right)  $, let $\chi_{K,\mathbf{y}%
}:\hat{K}\rightarrow K$ denote an affine bijection with $\chi_{K}:\hat
{F}\rightarrow F$ for $\hat{F}:=\left\{  \mathbf{x}\in\hat{K}\mid
x_{3}=0\right\}  $ and $\chi_{K,\mathbf{y}}\left(  1,0,0\right)  =\mathbf{y}$.
Hence, it suffices to show that
\begin{align*}
I_{k}^{
\boldsymbol{\mu}%
}  = 0, \qquad \forall%
\boldsymbol{\mu}%
\in\mathbb{N}_{0}^{3}\quad\text{with\quad}\left\vert
\boldsymbol{\mu}%
\right\vert \leq k-1,
\end{align*}
where
\begin{align*}
I_{k}^{
	\boldsymbol{\mu}%
} &:=\int_{0}^{1}\int_{0}^{1-x_{1}}Q_{k}\left(  1-2x_{1}\right)  \left(
1-x_{1}-x_{2}\right)  ^{\mu_{1}}x_{1}^{\mu_{2}}x_{2}^{\mu_{3}}dx_{2}dx_{1}\\
&  =\int_{0}^{1}Q_{k}\left(  1-2x_{1}\right)  g\left(  x_{1}\right)  dx_{1},\\
g\left(  x_{1}\right)  &:=x_{1}^{\mu_{2}}\left(  \int_{0}^{1-x_{1}}\left(
1-x_{1}-x_{2}\right)  ^{\mu_{1}}x_{2}^{\mu_{3}}dx_{2}\right)  .
\end{align*}
It is easy to see that $g\in\mathbb{P}_{\mu_{2}+\mu_{1}+\mu_{3}+1}%
\subset\mathbb{P}_{k}$ and $g\left(  1\right)  =0$. Therefore, 
we get%
\begin{align*}
I_{k}^{
\boldsymbol{\mu}
}  &  =\frac{1}{k+1}\int_{0}^{1}\left(  L_{k+1}-L_{k}\right)  ^{\prime}\left(
1-2x\right)  g(x)dx\\
&  =\frac{1}{2\left(  k+1\right)  }\int_{0}^{1}\left(  L_{k+1}-L_{k}\right)
\left(  1-2x\right)  \underbrace{g^{\prime}(x)}_{\in\mathbb{P}_{k-1}%
}dx-\left.  \frac{1}{2\left(  k+1\right)  }\left(  L_{k+1}-L_{k}\right)
\left(  1-2x\right)  g(x)\right\vert _{x=0}^{x=1}\\
&  =0,
\end{align*}
by the orthogonality of the Legendre polynomials and by the 
properties
$L_{n}(1)=1$ for all $n\in\mathbb{N}_{0}$ and $g(1)=0.$ This proves (a).

\textbf{(b): }Let $k\geq1$ be odd. Let $F\in\mathcal{F}_{\Omega}$ and
$K\subset\mathcal{T}_{F}$. The vertex of $K$ opposite to $F$ is denoted by
$\mathbf{z}$. Note that $\lambda_{K,\mathbf{z}}\neq0$ on any 
$F^{\prime
}\subset\partial\omega_{F}\cap K$. The proof of orthogonality follows from a
repetition of the arguments as in the proof of (a).%
\end{proof}

We now define the space
\[
B_{k}^{\operatorname{nc}}\left(  \mathcal{T}\right)  :=\left\{
\begin{array}
[c]{ll}%
\operatorname{span}\left\{  B_{k}^{\operatorname{CR},K}\mid K\in
\mathcal{T}\right\}  & \text{if }k\text{ is even,}\\
\operatorname{span}\left\{  B_{k}^{\operatorname{CR},F}\mid F\in
\mathcal{F}_{\Omega}\right\}  & \text{if }k\text{ is odd.}%
\end{array}
\right.
\]

\begin{definition}\label{DefCanCR}
The scalar \emph{basic Crouzeix-Raviart space} of order
$k$ for conforming simplicial finite element meshes $\mathcal{T}$ in 3D is
given by (see (\ref{DefSkprime}))%
\begin{equation}
\operatorname*{CR}\nolimits_{k,0}\left(  \mathcal{T}\right)  :=B_{k}%
^{\operatorname{nc}}\left(  \mathcal{T}\right)  +\left\{
\begin{array}
[c]{ll}%
S_{k,0}\left(  \mathcal{T}\right)  & \text{if }k\text{ is even,}\\
S_{k,0}^{\prime}\left(  \mathcal{T}\right)  & \text{if }k\text{ is odd.}%
\end{array}
\right.  \label{defCRcan}%
\end{equation}
The \emph{basic Crouzeix-Raviart Stokes element} is given by
$\mathbf{V}_{h}:=\mathbf{CR}_{k,0}\left(  \mathcal{T}\right)  :=\left(
\operatorname*{CR}_{k,0}\left(  \mathcal{T}\right)  \right)  ^{3}$ 
and
$M_{h}:=\mathbb{P}_{k-1,0}\left(  \mathcal{T}\right)  $. The 
norm in \(\mathbf{V}_{h}\) is given by the broken \(H^1\)-seminorm,
	\begin{align*}
	\|\mathbf{v}\|_{\mathbf{V}_{h}} := 
	\sqrt{\sum\limits_{K\in\mathcal{T}} \int_{K} 
		\left|\nabla\mathbf{v}\right|^2},
	\end{align*}
which is a norm owing to the discrete  Poincaré-Friedrichs inequality 
\cite[Theorem (10.6.12)]{scottbrenner}.
In this case, the
bilinear forms $a_{h}\left(  \cdot,\cdot\right)  $, $b_{h}\left(  \cdot
,\cdot\right)  $ in (\ref{Stokesdiscweak}) are given by%
\begin{equation*}%
\begin{array}
[c]{ll}%
a_{h}\left(  \mathbf{u},\mathbf{v}\right)  :=%
{\displaystyle\sum\limits_{K\in\mathcal{T}}}
{\displaystyle\int_{K}}
\left\langle \nabla\mathbf{u},\nabla\mathbf{v}\right\rangle  & \mathbf{\forall
u},\mathbf{v}\in\mathbf{V}_{h},\\
b_{h}\left(  q,\mathbf{v}\right)  :=%
{\displaystyle\sum\limits_{K\in\mathcal{T}}}
{\displaystyle\int_{K}}
q\operatorname{div}\mathbf{v} & \forall\left(  q,\mathbf{v}\right)  \in
M_{h}\times\mathbf{V}_{h}.
\end{array}
\end{equation*}

\end{definition}

\begin{remark}
For $k=1$ and $F\in\mathcal{F}_{\Omega}$, we obtain explicitly%
\[
B_{1}^{\operatorname*{CR},F}:=%
\begin{cases}
1-3\lambda_{K,\mathbf{z}} & \text{on }K\in\mathcal{T}_{F},\\
0 & \text{otherwise,}%
\end{cases}
\]
where $\mathbf{z}\in\mathcal{V}\left(  K\right)  $ is the vertex opposite to
$F$. This function is the basis function from the original paper
\cite[(5.1)]{CrouzeixRaviart}.
\end{remark}

\begin{lemma}
For any $k$, the sum in (\ref{defCRcan}) is direct and we have%
\[
S_{k,0}\left(  \mathcal{T}\right)  \subset\operatorname*{CR}\nolimits_{k,0}%
\left(  \mathcal{T}\right)  \subset\operatorname*{CR}\nolimits_{k,0}^{\max
}\left(  \mathcal{T}\right)  .
\]
\end{lemma}

\begin{proof}
The proof that the sum in (\ref{defCRcan}) for even $k$ is direct can be found
in \cite[Theorem 33]{CDS}. For odd $k$, the proof in \cite[Theorem 33]{CDS}
can be adapted: For $k=1$, we have $S_{1,0}^{\prime}\left(
\mathcal{T}\right)  =\left\{  0\right\}  $ and the statement is trivial. 

It
remains to consider odd $k\geq3$. We need to show that $S_{k,0}^{\prime
}\left(  \mathcal{T}\right)  \cap B_{k}^{\operatorname{nc}}\left(
\mathcal{T}\right)  =\left\{  0\right\}  $. Assume that $u\in
S_{k,0}^{\prime}\left(  \mathcal{T}\right)  \cap B_{k}^{\operatorname{nc}%
}\left(  \mathcal{T}\right)  $. By definition,
$u\in 
C^{0}(\Omega)$ and
$u$ vanishes on the boundary. Moreover, $u(\mathbf{z})=0$ for all
$\mathbf{z}\in\mathcal{V}$.
We now consider a tetrahedron $K\in\mathcal{T}$ with facets $F_{i}$ (opposite
to the vertices $\mathbf{z}_{i}\in\mathcal{V}\left(  K\right)  $, $1\leq
i\leq4$) which has one facet that lies on the boundary, say 
$F_{1}\in\mathcal{F}_{\partial\Omega}$. This
implies that $u$ vanishes on $F_{1}$. Since $u\in B_{k}^{\operatorname{nc}%
}\left(  \mathcal{T}\right)  $ for odd $k$ we have%
\[
\left.  u\right\vert _{K}=\sum_{i=2}^{4}\alpha_{i}Q_{k}\left(  1-2\lambda
_{K,\mathbf{z}_{i}}\right)  \quad\text{for some }\alpha_{2},\alpha_{3}%
,\alpha_{4}\in\mathbb{R}.
\]
This leads to the system
\begin{align*}
0=u\left(  \mathbf{z}_{2}\right)   &  =\alpha_{1}Q_{k}\left(  -1\right)
+\alpha_{2}Q_{k}\left(  1\right)  +\alpha_{3}Q_{k}\left(  1\right)  ,\\
0=u\left(  \mathbf{z}_{3}\right)   &  =\alpha_{1}Q_{k}\left(  1\right)
+\alpha_{2}Q_{k}\left(  -1\right)  +\alpha_{3}Q_{k}\left(  1\right)  ,\\
0=u\left(  \mathbf{z}_{4}\right)   &  =\alpha_{1}Q_{k}\left(  1\right)
+\alpha_{2}Q_{k}\left(  1\right)  +\alpha_{3}Q_{k}\left(  -1\right)  ,
\end{align*}
since $u\left(  \mathbf{z}_{i}\right)  =0$ for $i=2,3,4$.  Moreover, we
have (for odd $k$)
\begin{align*}
Q_{k}\left(  \pm1\right)   &  =\frac{1}{k+1}\left(  L_{k+1}-L_{k}\right)
^{\prime}\left(  \pm1\right)  \overset{\text{(\ref{Lendpoints})}}{=}\frac
{1}{k+1}\left(  \left(  \pm1\right)  ^{k}\binom{k+2}{2}-\left(  \pm1\right)
^{k-1}\binom{k+1}{2}\right) \\
&  =\left\{
\begin{array}
[c]{ll}
1 & \text{for }+1,\\
-\left(  k+1\right)  & \text{for }-1
\end{array}
\right.
\end{align*}
and, in turn, $\boldsymbol{\alpha}
=\left(  \alpha_{2},\alpha_{3},\alpha_{4}\right)^\intercal  $ is the solution 
of
\[
\left[
\begin{array}
[c]{ccc}%
-\left(  k+1\right)  & 1 & 1\\
1 & -\left(  k+1\right)  & 1\\
1 & 1 & -\left(  k+1\right)
\end{array}
\right]
\boldsymbol{\alpha}
=\left(
\begin{array}
[c]{c}%
0\\
0\\
0
\end{array}
\right)  .
\]
The determinant of this matrix is $-\left(  k-1\right)  \left(  k+2\right)
^{2}$; since $k\geq3$ the matrix is regular and $\boldsymbol{\alpha}
=\left(  0,0,0\right)  $ follows so that $\left.  u\right\vert _{K}=0$. From
an induction argument, we conclude that $u=0$.%
\end{proof}

\section{Inf-sup stability for the quadratic velocity space}
\label{InfSup2}

We start this section with some remarks on the macroelement technique (see
\cite{Stenberg_macro}, \cite{Stenberg_marco_new}) which we employ for the
analysis of the inf-sup stability.

\begin{definition}
\label{DefNspace}A macroelement $\mathcal{M}$ is a connected set of elements
$K\in\mathcal{T}$. For a macroelement $\mathcal{M}$, the spaces
$N_{k,\mathcal{M}}^{\operatorname*{CR}}$ and $N_{k,\mathcal{M}}$ are the
orthogonal complements in $\mathbb{P}_{k-1}\left(  \mathcal{M}\right)  $ of
the images $\operatorname*{div}\left(  \mathbf{CR}_{k,0}\left(  \mathcal{M}%
\right)  \right)  $ and $\operatorname*{div}\mathbf{S}_{k,0}\left(
\mathcal{M}\right)  $:%
\begin{align*}
N_{k,\mathcal{M}}^{\operatorname*{CR}}  &  =\left\{  p\in\mathbb{P}%
_{k-1}\left(  \mathcal{M}\right)  \mid\forall\mathbf{v}\in\mathbf{CR}%
_{k,0}\left(  \mathcal{M}\right)  :\left(  p,\operatorname*{div}%
\mathbf{v}\right)  _{L^{2}(\operatorname*{dom}\mathcal{M})}=0\right\}  ,\\
N_{k,\mathcal{M}}  &  =\left\{  p\in\mathbb{P}_{k-1}\left(  \mathcal{M}%
\right)  \mid\forall\mathbf{v}\in\mathbf{S}_{k,0}\left(  \mathcal{M}\right)
:(p,\operatorname*{div}\mathbf{v})_{L^{2}(\operatorname*{dom}\mathcal{M}%
)}=0\right\}  .
\end{align*}
Non-zero elements in $N_{k,\mathcal{M}}\cap L_{0}^{2}\left(  \Omega\right)  $
are \emph{critical pressures} for $\mathcal{M}$.
\end{definition}

A direct consequence of \cite[Thm. 2.1]{Stenberg_marco_new} is the following
proposition for quadratic velocity spaces.

\begin{proposition}
\label{thm:macroelement} Let $\mathcal{T}$ be a regular finite element
simplicial mesh on a bounded polyhedral domain $\Omega\subset\mathbb{R}^{3}$.
Let $k=2$ and consider the macroelements consisting of one simplex, i.e.
$\mathcal{M}=\left\{  K\right\}  \subset\mathcal{T}$. If
\begin{equation}
\dim N_{2,\left\{  K\right\}  }^{\operatorname*{CR}}=1,\qquad\forall
K\in\mathcal{T}, \label{N2Kdim}
\end{equation}
then the basic Crouzeix-Raviart element for the Stokes problem (cf. Def.
\ref{DefCanCR}) is inf-sup stable.
\end{proposition}

In the remaining part of this section, we prove the inf-sup stability of
$\left(  \mathbf{CR}_{2,0}\left(  \mathcal{T}\right)  ,\mathbb{P}_{1,0}\left(
\mathcal{T}\right)  \right)  $ for the case $k=2$ by showing (\ref{N2Kdim}).
Using $P_{1}^{\left(  0,3\right)  }\left(  x\right)  
=\left(  5x-3\right)
/2$ and $P_{1}^{\left(  0,3\right)  }\left(  1-2\lambda_{K,\mathbf{z}}\right)
=1-5\lambda_{K,\mathbf{z}}$, it is easy to verify that the 
polynomials
\[
P_{1}^{\left(  0,3\right)  }\left(  1-2\lambda_{K,\mathbf{z}}\right)
,\quad\mathbf{z}\in\mathcal{V}\left(  K\right)
\]
form a basis for $\mathbb{P}_{k-1}(K)=\mathbb{P}_{1}(K)$.

The following lemma will be used in Theorem \ref{thm:stabilityk2} and can 
be 
easily proved
using a transformation to the reference tetrahedron and the orthogonality
properties of $P_{1}^{\left(  0,3\right)  }(x)$ (see Lemma 
\ref{LemOrthoMain}).

\begin{lemma}
\label{lem:orthog}For $\mathbf{y},\mathbf{z}\in\mathcal{V}\left(  K\right)  $,
it holds
\begin{equation}
\int_{K}P_{1}^{\left(  0,3\right)  }\left(  1-2\lambda_{K,\mathbf{y}}\right)
\lambda_{K,\mathbf{z}}=\left\{
\begin{array}
[c]{ll}%
-\frac{\left\vert K\right\vert }{4} & \mathbf{y}=\mathbf{z},\\
0 & \mathbf{y}\neq\mathbf{z}.
\end{array}
\right.  \label{P0,3prop}%
\end{equation}

\end{lemma}

The goal of this section is to prove the stability of the pair $\left(
\mathbf{CR}_{2,0}\left(  \mathcal{T}\right)  ,\mathbb{P}_{1,0}\left(
\mathcal{T}\right)  \right)  $ using the macroelement technique (Proposition
\ref{thm:macroelement}), where each macroelement consists of a single
tetrahedron $K\in\mathcal{T}$.

\begin{theorem}\label{thm:stabilityk2}
The basic Crouzeix-Raviart Stokes element for \(k = 2\) is inf-sup 
stable 
with an inf-sup constant \(\gamma>0\) independent of the mesh width \(h\).
\end{theorem}

\begin{proof}

By a pullback to the reference element, it is straightforward to verify that
for $k=2$ the Crouzeix Raviart basis functions for a tetrahedron satisfy%
\[
B_{2}^{\operatorname*{CR},K}=\frac{5}{3}\left(  \sum_{\mathbf{y}\in
\mathcal{V}\left(  K\right)  }L_{2}\left(  1-2\lambda_{K,\mathbf{y}}\right)
-1\right)  .
\]
By definition, 
$B_{2}^{\operatorname*{CR},K}\mathbf{t}\in\mathbf{CR}_{k,0}%
(\mathcal{T})$ for any constant vector $\mathbf{t}\in\mathbb{R}^{3}$. We will
show that the space
\[
N_{2,K}^{\operatorname*{CR}}=\left\{  p\in\mathbb{P}_{1}\left(  K\right)
~\Big|~\forall\mathbf{t}\in\mathbb{R}^{3}:\left(  
p,\operatorname*{div}%
B_{2}^{\operatorname*{CR},K}\mathbf{t}\right)  _{L^{2}(K)}=0\right\}
\]
is one-dimensional. Using the relation (see \cite[8.9.15]{NIST:DLMF})%
\[
L_{n}^{\prime}=\frac{n+1}{2}P_{n-1}^{\left(  1,1\right)  },\qquad\forall
n\in\mathbb{N}_{0},
\]
we get for $\mathbf{y}\in\mathcal{V}\left(  K\right)  $%
\begin{align}
\int_{K}P_{1}^{\left(  0,3\right)  }\left(  1-2\lambda_{K,\mathbf{y}}\right)
&\operatorname{div}\left(  B_{2}^{\operatorname*{CR},K}\mathbf{t}\right) 
=-5\sum_{\mathbf{z}\in\mathcal{V}\left(  K\right)  }\partial_{\mathbf{t}%
}\lambda_{K,\mathbf{z}}\int_{K}P_{1}^{\left(  0,3\right)  }\left(
1-2\lambda_{K,\mathbf{y}}\right)  P_{1}^{\left(  1,1\right)  }\left(
1-2\lambda_{K,\mathbf{z}}\right)  . \label{pki1}%
\end{align}
Let $\chi_{K}:\hat{K}\rightarrow K$ denote an affine pullback and
$\mathbf{\hat{y}}:=\chi_{K}^{-1}\left(  \mathbf{y}\right)  $, $\mathbf{\hat
{z}}:=\chi_{K}^{-1}\left(  \mathbf{z}\right)  $, and $\lambda_{\mathbf{\hat
{z}}}:=\lambda_{K,\mathbf{z}}\circ\chi_{K}$. Then,%
\begin{align}
\begin{split}
\int_{K}P_{1}^{\left(  0,3\right)  }\left(  1-2\lambda_{K,\mathbf{y}}\right)
P_{1}^{\left(  1,1\right)  }\left(  1-2\lambda_{K,\mathbf{z}}\right)   &
=\frac{\left\vert K\right\vert }{\left\vert \hat{K}\right\vert }\int_{\hat{K}%
}P_{1}^{\left(  0,3\right)  }\left(  1-2\lambda_{\mathbf{\hat{y}}}\right)
P_{1}^{\left(  1,1\right)  }\left(  1-2\lambda_{\mathbf{\hat{z}}}\right)
\\
&  =6\left\vert K\right\vert \int_{\hat{K}}P_{1}^{\left(  0,3\right)  }\left(
1-2\lambda_{\mathbf{\hat{y}}}\right)  P_{1}^{\left(  1,1\right)  }\left(
1-2\lambda_{\mathbf{\hat{z}}}\right)  .\label{pki2}
\end{split}
\end{align}
For a tetrahedron $K$, we fix a vertex $\mathbf{p}\in\mathcal{V}\left(
K\right)  $ and set $\mathbf{t}_{\mathbf{v}}:=\mathbf{v}-\mathbf{p}$,
$\mathbf{v}\in\mathcal{V}\left(  K\right)  \backslash\left\{  \mathbf{p}%
\right\}  $. Then, it is easy to verify that%
\[
\partial_{\mathbf{t}_{\mathbf{v}}}\lambda_{K,\mathbf{z}}=\left\{
\begin{array}
[c]{ll}%
-1 & \mathbf{z}=\mathbf{p},\\
\delta_{\mathbf{v},\mathbf{z}} & \mathbf{z}\in\mathcal{V}\left(  K\right)
\backslash\left\{  \mathbf{p}\right\},
\end{array}
\right.  \quad\text{for all }\mathbf{v}\in\mathcal{V}\left(  K\right)
\backslash\left\{  \mathbf{p}\right\}  .
\]
We combine this with (\ref{pki1}), (\ref{pki2}), and obtain for any
$\mathbf{v}\in\mathcal{V}\left(  K\right)  \backslash\left\{  \mathbf{p}%
\right\}  $%
\begin{align*}
  \int_{K}P_{1}^{\left(  0,3\right)  }\left(  1-2\lambda_{K,\mathbf{y}%
}\right) & \operatorname{div}\left(  B_{2}^{\operatorname*{CR},K}%
\mathbf{t}_{\mathbf{v}}\right) 
=-30\left\vert K\right\vert\!\!\sum
_{\mathbf{z}\in\mathcal{V}\left(  K\right)  }\partial_{\mathbf{t}_{\mathbf{v}%
}}\lambda_{K,\mathbf{z}}\int_{\hat{K}}P_{1}^{\left(  0,3\right)  }\left(
1-2\lambda_{\mathbf{\hat{y}}}\right)  P_{1}^{\left(  1,1\right)  }\left(
1-2\lambda_{\mathbf{\hat{z}}}\right) \\
& =30\left\vert K\right\vert \left(  \int_{\hat{K}}P_{1}^{\left(
0,3\right)  }\left(  1-2\lambda_{\mathbf{\hat{y}}}\right)  \left(
P_{1}^{\left(  1,1\right)  }\left(  1-2\lambda_{\mathbf{\hat{p}}}\right)
-P_{1}^{\left(  1,1\right)  }\left(  1-2\lambda_{\mathbf{\hat{v}}}\right)
\right)  \right)  .
\end{align*}
The difference in the integrand can be simplified by using
$P_{1}^{\left(  1,1\right)  }\left(  x\right)  =2x$:%
\[
P_{1}^{\left(  1,1\right)  }\left(  1-2\lambda_{\mathbf{\hat{p}}}\right)
-P_{1}^{\left(  1,1\right)  }\left(  1-2\lambda_{\mathbf{\hat{v}}}\right)
=4\left(  \lambda_{\mathbf{\hat{v}}}-\lambda_{\mathbf{\hat{p}}}\right)  .
\]
Using Lemma \ref{lem:orthog}, we have for any $\mathbf{v}\in\mathcal{V}\left(
K\right)  \backslash\left\{  \mathbf{p}\right\}  $%
\begin{align*}
\int_{K}P_{1}^{\left(  0,3\right)  }\left(  1-2\lambda_{K,\mathbf{y}}\right)
\operatorname{div}(B_{2}^{\operatorname*{CR},K}\mathbf{t}_{\mathbf{v}})  &
=120\left\vert K\right\vert \int_{\hat{K}}P_{1}^{\left(  0,3\right)  }\left(
1-2\lambda_{\mathbf{\hat{y}}}\right)  \left(  \lambda_{\mathbf{\hat{v}}%
}-\lambda_{\mathbf{\hat{p}}}\right) \\
=  &
\begin{cases}
-120\left\vert K\right\vert \int_{\hat{K}}P_{1}^{\left(  0,3\right)  }\left(
1-2\lambda_{\mathbf{\hat{y}}}\right)  \lambda_{\mathbf{\hat{y}}} &
\mathbf{p}=\mathbf{y},\\
120\left\vert K\right\vert \int_{\hat{K}}P_{1}^{\left(  0,3\right)  }\left(
1-2\lambda_{\mathbf{\hat{y}}}\right)  \lambda_{\mathbf{\hat{y}}} &
\mathbf{p}\neq\mathbf{y\quad}\text{and\quad}\mathbf{v}=\mathbf{y},\\
0 & \text{otherwise}%
\end{cases}
\\
\overset{\text{(\ref{P0,3prop})}}{=}  &
\begin{cases}
5\left\vert K\right\vert  & \mathbf{p}=\mathbf{y},\\
-5\left\vert K\right\vert  & \mathbf{p}\neq\mathbf{y\quad}\text{and\quad
}\mathbf{v}=\mathbf{y},\\
0 & \text{otherwise.}%
\end{cases}
\end{align*}
This implies that the matrix $\left(  \left(  \operatorname*{div}%
B_{2}^{\operatorname*{CR},K}\mathbf{t}_{\mathbf{v}},P_{1}^{\left(  0,3\right)
}\left(  1-2\lambda_{K,\mathbf{y}}\right)  \right)  _{L^{2}\left(  K\right)
}\right)  _{\substack{\mathbf{v}\in\mathcal{V}\left(  K\right)  \backslash
\left\{  \mathbf{p}\right\}  \\\mathbf{y}\in\mathcal{V}\left(  K\right)  }}$
is given by%
\[
5\left\vert K\right\vert
\begin{pmatrix}
1 & -1 & 0 & 0\\
1 & 0 & -1 & 0\\
1 & 0 & 0 & -1
\end{pmatrix}
.
\]
This matrix has full rank so that $N_{2,\{K\}}^{\operatorname*{CR}}$ is
one-dimensional, containing only the constant pressures. Inf-sup stability 
follows by the macroelement technique (Proposition \ref{thm:macroelement}).
\end{proof}

\begin{remark}
If $\mathcal{F}_{\Omega}\neq\emptyset$, then $\mathbf{CR}_{2,0}\left(
\mathcal{T}\right)  \subsetneq\mathbf{CR}_{2,0}^{\max}\left(  \mathcal{T}%
\right)  $ since $\mathbf{CR}_{2,0}^{\max}\left(  \mathcal{T}\right)  $ also
contains one non-conforming Crouzeix-Raviart functions associated with each
facet $F\in\mathcal{F}_{\Omega}$ (see Remark \ref{Remmaxcan}).

We have shown that it suffices to enrich $\mathbf{S}_{2,0}(\mathcal{T})$ with 
certain local
Crouzeix-Raviart functions (one scalar one per tetrahedron multiplied by three
linearly independent constant vectors in $\mathbb{R}^{3}$) in order to
stabilize the pair $\left(  \mathbf{S}_{2,0}(\mathcal{T}),\mathbb{P}%
_{1,0}\left(  \mathcal{T}\right)  \right)  $.
\end{remark}
\section{Critical pressures for \(\left(  \mathbf{S}_{k,0}\left(
\mathcal{T}\right)  ,\mathbb{P}_{k-1,0}\left(  \mathcal{T}\right)  \right)  \)
in 3D}
\label{SecCritConf}

In two dimensions, the existence of critical
pressures (or also called 
\emph{spurious} pressures in the literature) is 
related to the 
existence of critical vertices
\cite{ScottVogelius}. In this case, a vertex $\mathbf{z}$ is
\textit{critical} if all edges connected to $\mathbf{z}$ lie on at most two
straight lines; in \cite{ScottVogelius} the following dimension formula is
proved for $k\geq4$ and two-dimensional triangulations (recall Def.
\ref{DefNspace})
\begin{equation}
\dim N_{k,\mathcal{T}}=1+\#\left\{  \mathbf{z}^{\prime}\in\mathcal{V}_{\Omega
}\mid\mathbf{z}^{\prime}\text{ is a critical vertex in }\mathcal{T}\right\}
. \label{eq:dimformula}%
\end{equation}
In \cite[Def. 6.3]{neilan2015discrete}, \textit{critical edges} for conforming
simplicial meshes in dimension $3$ are introduced in analogy to
\textit{critical vertices} in two dimensions.

\begin{definition}
\label{CritEdge}Let $\mathcal{T}$ be a conforming simplicial finite element
mesh and let $\mathcal{E}$ denote the set of edges in $\mathcal{T}$. An edge
$E\in\mathcal{E}$ is \emph{critical in }$\mathcal{T}$ if all facets
$F\in\mathcal{F}$ with $E\subset F$ lie in at most two flat planes.
\end{definition}

An analogous dimension formula to (\ref{eq:dimformula}) is not known for
conforming simplicial meshes in dimension $d\geq3$. In this section, we
discuss the existence of critical pressures in the presence of \emph{critical
edges} (cf. Def. \ref{CritEdge}). In Section \ref{SecCRstab}, we then prove
that the non-conforming Crouzeix-Raviart functions remove these critical 
pressures.

\begin{remark}
At the current stage of research, a complete description of all critical
pressures is still open and, hence, our result can be interpreted in the way
that an inf-sup stable Crouzeix-Raviart space should contain the basic
Crouzeix-Raviart space while the inf-sup stability of basic
Crouzeix-Raviart elements for the Stokes equation is still not fully understood.
\end{remark}

There are exactly two types of critical edges that can 
appear 
in a
tetrahedral mesh.

\begin{enumerate}
\item[a.] The critical edge $E\in\mathcal{E}_{\Omega}$ is an inner edge of
$\mathcal{T}$. In this case the edge patch $\omega_{E}$ consists of exactly
four tetrahedra (see Fig. \ref{fig:4patch} for an illustration).
\begin{figure}[ptb]
\begin{center}
	\includegraphics[scale=0.7]{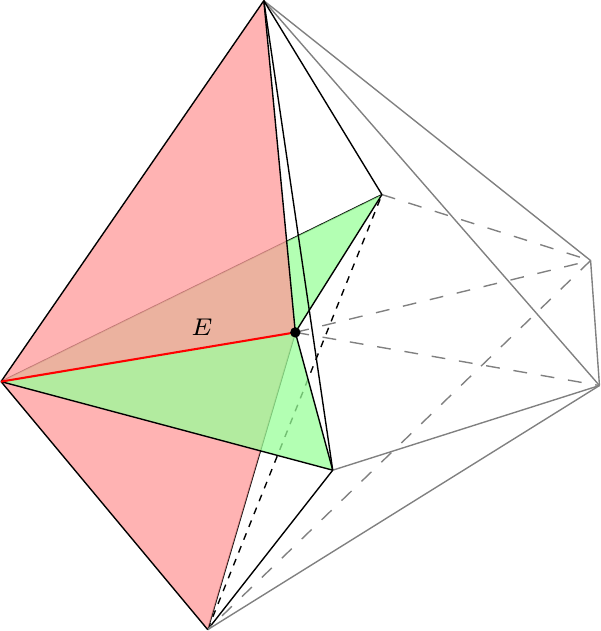}
\end{center}
\caption{Inner critical edge $E$ in a nodal patch. Red 
and green 
coloured facets lie in one plane, respectively.}
\label{fig:4patch}%
\end{figure}

\item[b.] The critical edge $E\in\mathcal{E}_{\partial\Omega}$ is an outer
edge of $\mathcal{T}$ and the edge patch $\omega_{E}$ consists of either 
one, two
or three tetrahedra (see Fig. \ref{fig:edgepatch} for an illustration).
\begin{figure}[ptb]
\begin{center}
	\includegraphics[scale=0.7]{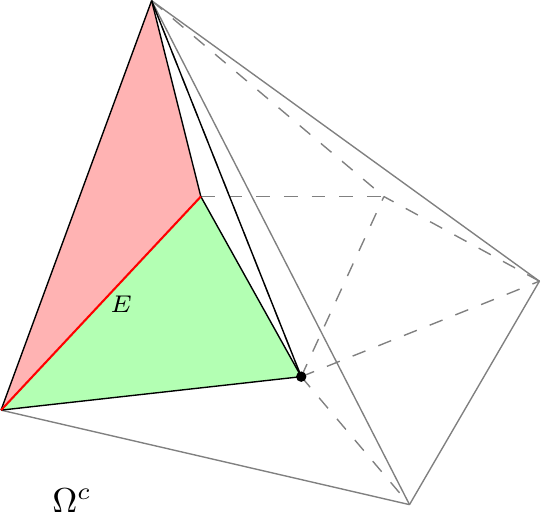}\qquad
	\includegraphics[scale=0.7]{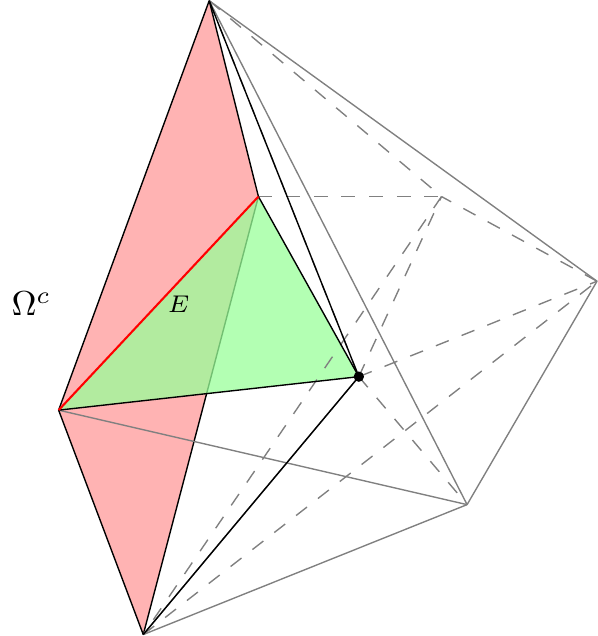}\qquad
	\includegraphics[scale=0.7]{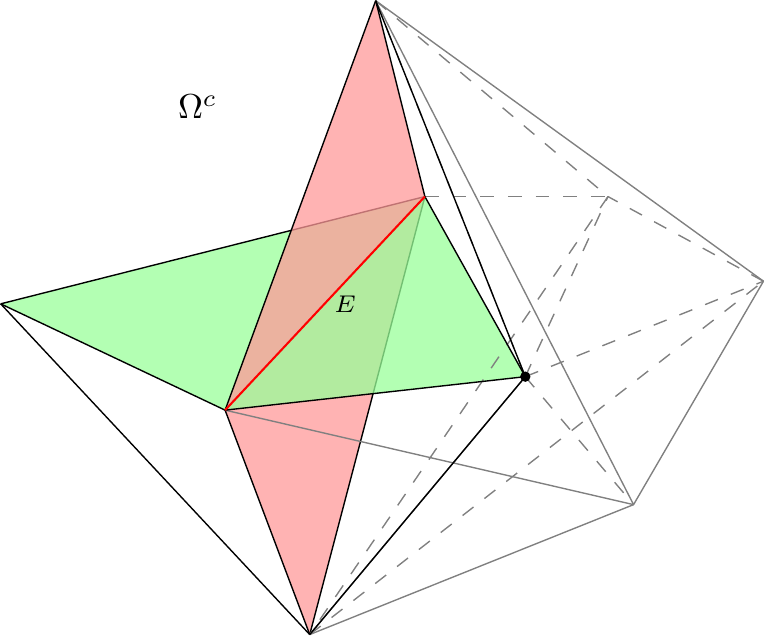}
\end{center}
\caption{The three cases of possible critical edges $E$ at the boundary of 
a 
mesh. The edge patch \(\omega_E\) consist of one, two or three 
tetrahedra (from left to right). Red and green coloured
facets lie in one plane, respectively.}%
\label{fig:edgepatch}%
\end{figure}
\end{enumerate}

We now study the existence of critical pressures for those two
types of critical edges.

\subsection{Critical pressure for inner critical 
edges}\label{sect:criticalpressure-inner}

Let $E\in\mathcal{E}_{\Omega}$ be a critical edge of the mesh $\mathcal{T}$.
In the following, we construct pressures $p_{k-1}^{E,\mathbf{p}}%
\in\mathbb{P}_{k-1,0}(\mathcal{T})$, $\mathbf{p}\in\mathcal{V}\left(
E\right)  $, supported on $\omega_{E}$ such that
\begin{equation}
\left(  p_{k-1}^{E,\mathbf{p}},\operatorname{div}\mathbf{v}\right)
_{L^{2}\left(  \Omega\right)  }=0\quad\forall\mathbf{v}\in\mathbf{S}%
_{k,0}\left(  \mathcal{T}\right)  . \label{zerocond}%
\end{equation}

\begin{notation}
\label{NotE}For a set $M\subset\mathbb{R}^{3}$, we denote its convex hull by
$\left[  M\right]  $. Let $E\in\mathcal{E}_{\Omega}$ be a critical 
edge. This implies
that $E$ is shared by exactly four tetrahedra $K_{j}\in\mathcal{T}$, $1\leq 
j\leq4$ (see Figure \ref{fig:4patch}). We employ a cyclic numbering convention 
and
denote $K_{4+1}:=K_{1}$ and $K_{1-1}:=K_{4}.$ We assume the numbering
convention that $K_{i}$, $K_{i+1}$ share a facet (denoted by $F_{i}$) and
$F_{1},F_{3}$ lie in one plane and $F_{2},F_{4}$ lie in one other plane.
\end{notation}

\begin{definition}
[Critical pressure for inner edges]\label{PressCritic}Assume the setting as in
Notation \ref{NotE} and let $k\geq1$. Assume that there is some critical edge
$E\in\mathcal{E}_{\Omega}$. For $\mathbf{p}\in\mathcal{V}\left(  E\right)  $,
the \emph{critical pressure} $p_{k-1}^{E,\mathbf{p}}\in\mathbb{P}%
_{k-1,0}\left(  \mathcal{T}\right)  $ is given by%
\[
p_{k-1}^{E,\mathbf{p}}:=\sum_{i=1}^{4}\frac{\left(  -1\right)  ^{i}%
}{\left\vert K_{i}\right\vert }\chi_{K_{i}}P_{k-1}^{\left(  0,3\right)
}\left(  1-2\lambda_{K_{i},\mathbf{p}}\right)  ,
\]
where $\chi_{K_{i}}$ denotes the characteristic function for $K_{i}$.
\end{definition}

Next, we verify that (\ref{zerocond}) is satisfied. The restriction of
$\mathbf{S}_{k,0}\left(  \mathcal{T}\right)  $ to $K\in\mathcal{T}$ belongs to
the span of
\[%
 \boldsymbol{\lambda}%
_{K}^{%
\boldsymbol{\mu}%
}\mathbf{w}_{\ell}\qquad\left\{
\begin{array}
[c]{l}%
\text{for }\mathbf{w}_{\ell}\in\mathbb{R}^{3},\ 1\leq\ell\leq
3,\ \text{linearly independent,}\quad
\\
\boldsymbol{\mu}%
=\left(  \mu_{\mathbf{v}}\right)  _{\mathbf{v}\in\mathcal{V}\left(  K\right)
}\in\mathbb{N}_{0}^{4}\text{, }\left\vert
\boldsymbol{\mu}%
\right\vert =k.
\end{array}
\right.
\]
For $\mathbf{p}\in\mathcal{V}\left(  E\right)  $, we have for $K_{i}$ as in
Notation \ref{NotE}%
\begin{align}
&  \int_{K_{i}}p_{k-1}^{E,\mathbf{p}}\operatorname{div}\left(
 \boldsymbol{\lambda}%
_{K_{i}}^{%
\boldsymbol{\mu}%
}\mathbf{w}_{\ell}\right)  =\int_{K_{i}}p_{k-1}^{E,\mathbf{p}}\partial
_{\mathbf{w}_{\ell}}%
 \boldsymbol{\lambda}%
_{K_{i}}^{%
\boldsymbol{\mu}%
}\nonumber  =\frac{\left(  -1\right)  ^{i}}{\left\vert K_{i}\right\vert 
}\int_{K_{i}%
}P_{k-1}^{\left(  0,3\right)  }\left(  1-2\lambda_{K_{i},\mathbf{p}}\right)
\partial_{\mathbf{w}_{\ell}}%
 \boldsymbol{\lambda}%
_{K_{i}}^{%
\boldsymbol{\mu}%
}\nonumber\\
&  =\frac{\left(  -1\right)  ^{i}}{\left\vert K_{i}\right\vert }%
\sum_{\mathbf{v}\in\mathcal{V}\left(  K_{i}\right)  }\mu_{\mathbf{v}}%
\partial_{\mathbf{w}_{\ell}}\lambda_{K_{i},\mathbf{v}}\int_{K_{i}}%
P_{k-1}^{\left(  0,3\right)  }\left(  1-2\lambda_{K_{i},\mathbf{p}}\right)
 \boldsymbol{\lambda}%
_{K_{i}}^{%
\boldsymbol{\mu}%
-\mathbf{e}_{\mathbf{v}}^{K_{i}}}\nonumber\\
&  =\frac{\left(  -1\right)  ^{i}}{\left\vert K_{i}\right\vert }\left(
\int_{K_{i}}P_{k-1}^{\left(  0,3\right)  }(1-2\lambda_{K_{i},\mathbf{p}%
})\lambda_{K_{i},\mathbf{p}}^{k-1}\right)  \times\left\{
\begin{array}
[c]{ll}%
k\partial_{\mathbf{w}_{\ell}}\lambda_{K_{i},\mathbf{p}} &
\boldsymbol{\mu}%
=k\mathbf{e}_{p}^{K_{i}},\\
\partial_{\mathbf{w}_{\ell}}\lambda_{K_{i},\mathbf{y}} & \left\{
\begin{array}
[c]{l}%
\boldsymbol{\mu}%
=\left(  k-1\right)  \mathbf{e}_{p}^{K_{i}}+\mathbf{e}_{\mathbf{y}}^{K_{i}},\\
\mathbf{y}\in\mathcal{V}\left(  K_{i}\right)  \backslash\left\{
\mathbf{p}\right\}  ,
\end{array}
\right. \\
0 & \text{otherwise,}%
\end{array}
\right.  \label{eq:intTi}%
\end{align}
where the last equality follows by the orthogonality of 
$P^{\left(  0,3\right)
}(1-2\lambda_{K_{i},\mathbf{p}})$ with respect to the weights $\lambda
_{K_{i},\mathbf{y}}$, $\mathbf{y}\in\mathcal{V}\left(  K\right)
\backslash\left\{  \mathbf{p}\right\}  $ (cf. Lem. \ref{LemOrthoMain}).

\begin{remark}
\label{RemTypes}We distinguish between four types of basis functions of
$\mathbf{S}_{k,0}\left(  \mathcal{T}\right)  $:

\begin{enumerate}[(a)]
\item basis function supported on one tetrahedron $K\in\mathcal{T}$, 
\label{casea}
\item basis functions supported on $\omega_{F}$ for some facet
$F\in\mathcal{F}_{\Omega}$ (which do not belong to the span of basis functions
of Type a),\label{caseb}
\item basis functions supported on $\omega_{E^{\prime}}$ for some edge
$E^{\prime}\in\mathcal{E}_{\Omega}$ (which do not belong to the span of basis
functions of Type a), b).\label{casec}
\item basis functions associated with a node $\mathbf{z}\in\mathcal{V}%
_{\Omega}$.\label{cased}
\end{enumerate}
\end{remark}

Note that only basis functions which are not identically zero on $\omega_{E}$
are relevant. Hence, it suffices to evaluate the integral $\int_{K_{i}}%
p_{k-1}^{E,\mathbf{p}}\operatorname{div}\mathbf{b}$ for relevant basis
functions \(\mathbf{b}\) of these different types.

\textbf{Basis functions of Type (a)}: These basis functions only exist for
$k\geq4$. Let $K$ be a tetrahedron in the patch $\mathcal{T}_{E}$. In
this case the basis functions are given by
\[
\mathbf{B}_{K}^{\boldsymbol{\mu},\ell}=
 \boldsymbol{\lambda}_{K}^{\boldsymbol{\mu}+\mathbf{1}_{K}}\mathbf{w}_{\ell}
 \qquad \boldsymbol{\mu}
 =\left(  \mu_{\mathbf{y}}\right)  _{\mathbf{y}\in\mathcal{V}\left(  K\right)
}\in\mathbb{N}_{0}^{4},\ \left\vert
\boldsymbol{\mu}\right\vert =k-4,
\]
where $\left\{  \mathbf{w}_{\ell},1\leq\ell\leq3\right\}  $ is a basis in
$\mathbb{R}^{3}$. From (\ref{eq:intTi}) we conclude that
\[
\int_{\Omega}p_{k-1}^{E,\mathbf{p}}\operatorname{div}\mathbf{B}_{K}^{%
\boldsymbol{\mu}%
,\ell}=\int_{K}p_{k-1}^{E,\mathbf{p}}\operatorname{div}\mathbf{B}_{K}^{%
\boldsymbol{\mu}%
,\ell}=0.
\]

\textbf{Basis functions of Type (b):} Let $K$, $K^{\prime}$ be two adjacent
tetrahedra with common facet $F$. Then the basis functions of this type are
given by (cf. Notation \ref{NotBary})%
\[
\mathbf{B}_{F}^{%
\boldsymbol{\mu}%
,\ell}=\mathbf{w}_{\ell}%
\begin{cases}%
 \boldsymbol{\lambda}%
_{K,F}^{%
\boldsymbol{\mu}%
+\mathbf{1}_{F}} & \text{on }K,\\%
 \boldsymbol{\lambda}%
_{K^{\prime},F}^{%
\boldsymbol{\mu}%
+\mathbf{1}_{F}} & \text{on }K^{\prime},\\
0 & \text{otherwise},
\end{cases}
\]
for $
\boldsymbol{\mu}%
=\left(  \mu_{\mathbf{y}}\right)  _{\mathbf{y}\in\mathcal{V}\left(  F\right)
}\in\mathbb{N}_{0}^{3},\ \left\vert
\boldsymbol{\mu}%
\right\vert =k-3$. Similar as before, we conclude that in this case
\[
\int_{\Omega}p_{k-1}^{E,\mathbf{p}}\operatorname{div}\mathbf{B}_{F}^{%
\boldsymbol{\mu}%
,\ell}=\int_{\omega_{F}\cap\omega_{E}}p_{k-1}^{E,\mathbf{p}}\operatorname{div}%
\mathbf{B}_{F}^{%
\boldsymbol{\mu}%
,\ell}=0.
\]

\textbf{Basis functions of Type (c): }Recall that $p_{k-1}^{E,\mathbf{p}}$ 
has
support on the edge patch $\omega_{E}$ with respect to the critical edge $E$.
The basis functions associated to an edge $E^{\prime}\in\mathcal{E}_{\Omega}$,
have support $\omega_{E^{\prime}}$ and are defined, for $K^{\prime}%
\in\mathcal{T}_{E^{\prime}}$ by%
\[
\left.  \mathbf{B}{_{E^{\prime}}^{%
\boldsymbol{\mu}%
,\ell}}\right\vert _{{K^{\prime}}}=%
 \boldsymbol{\lambda}%
_{K^{\prime},E^\prime}^{%
\boldsymbol{\mu}%
+\mathbf{1}_{E^\prime}}\mathbf{w}_{\ell},
\]
where $
\boldsymbol{\mu}
=\left(  \mu_{\mathbf{y}}\right)  _{\mathbf{y}\in\mathcal{V}\left(  E\right)
}\in\mathbb{N}_{0}^{2}$ with $\left\vert
\boldsymbol{\mu}
\right\vert =k-2$. Hence, it holds
\begin{equation}
\int_{\Omega}p_{k-1}^{E,\mathbf{p}}\operatorname{div}\mathbf{B}_{E^{\prime}}^{%
\boldsymbol{\mu}%
,\ell}=\int_{\omega_{E}\cap\omega_{E^{\prime}}}p_{k-1}^{E,\mathbf{p}%
}\operatorname{div}\mathbf{B}_{E^{\prime}}^{%
\boldsymbol{\mu}%
,\ell}. \label{intedge}%
\end{equation}
Since $E$ is an inner edge, the edge patch $\omega_{E}$ 
consists of four
tetrahedra and therefore $\omega_{E}\cap\omega_{E}^{\prime}$
\begin{enumerate}
\item[(i)] is the union of two tetrahedra, or

\item[(ii)] is $\omega_{E}$ (for $E=E^{\prime}$).
\end{enumerate}
We discuss these two cases separately.

\textbf{Case (i):} Without loss of generality we assume $\omega_{E}\cap
\omega_{E^{\prime}}=K_{1}\cup K_{2}$ (cf. Notation \ref{NotE}). This implies
$E^{\prime}=\left[  \mathbf{z},\mathbf{q}\right]  $ for $\mathbf{z}%
\in\mathcal{V}\left(  E\right)  $ and $\mathbf{q}\notin\mathcal{V}\left(
E\right)  $. Then%
\[
\int_{K_{1}\cup K_{2}}p_{k-1}^{E,\mathbf{p}}\operatorname{div}B_{E^{\prime}}^{%
\boldsymbol{\mu}%
,\ell}=\sum_{i=1}^{2}\frac{\left(  -1\right)  ^{i}}{\left\vert K_{i}%
\right\vert }\int_{K_{i}}P_{k-1}^{\left(  0,3\right)  }(1-2\lambda
_{K_{i},\mathbf{p}})\partial_{\mathbf{w}_{\ell}}\left(
 \boldsymbol{\lambda}%
_{K_{i},E^{\prime}}^{%
\boldsymbol{\mu}%
+\mathbf{1}_{E^\prime}}\right)  .
\]
In view of (\ref{eq:intTi}), this integral can be different from zero only if
$\mathbf{p}=\mathbf{z}$ and $%
\boldsymbol{\mu}%
=\left(  k-2\right)  \mathbf{e}_{\mathbf{z}}^{E^{\prime}}$. For $%
\boldsymbol{\mu}%
=\left(  k-2\right)  \mathbf{e}_{\mathbf{z}}^{E^{\prime}}$ and $\widetilde{%
\boldsymbol{\mu}%
}:=\left(  k-1\right)  \mathbf{e}_{\mathbf{z}}^{E^{\prime}}+\mathbf{e}%
_{\mathbf{q}}^{E^{\prime}}$, we conclude from (\ref{eq:intTi}) that%
\begin{align*}
\int_{K_{1}\cup K_{2}}p_{k-1}^{E,\mathbf{z}}\operatorname{div}B_{E^{\prime}}^{%
\boldsymbol{\mu}%
,\ell}  &  =\sum_{i=1}^{2}\frac{\left(  -1\right)  ^{i}}{\left\vert
K_{i}\right\vert }\int_{K_{i}}P_{k-1}^{\left(  0,3\right)  }(1-2\lambda
_{K_{i},\mathbf{z}})\partial_{\mathbf{w}_{\ell}}\left(
 \boldsymbol{\lambda}%
_{K_{i}}^{\widetilde{%
\boldsymbol{\mu}
}}\right)\\
&  =\sum_{i=1}^{2}\frac{\left(  -1\right)  ^{i}}{\left\vert K_{i}\right\vert
}\int_{K_{i}}P_{k-1}^{\left(  0,3\right)  }(1-2\lambda_{K_{i},\mathbf{z}%
})\lambda_{K_{i},\mathbf{z}}^{k-1}\partial_{\mathbf{w}_{\ell}}\lambda
_{K_{i},\mathbf{q}}.
\end{align*}
A pullback to the reference element leads to%
\begin{equation}
\int_{K_{1}\cup K_{2}}p_{k-1}^{E,\mathbf{z}}\operatorname{div}B_{E^{\prime}}^{%
\boldsymbol{\mu}
,\ell}=-\frac{1}{|\hat{K}|}\int_{\hat{K}}P_{k-1}^{\left(  0,3\right)
}(1-2\lambda_{\mathbf{\hat{z}}})\lambda_{\mathbf{\hat{z}}}^{k-1}\left(
\partial_{\mathbf{w}_{\ell}}\lambda_{K_{1},\mathbf{q}}-\partial_{\mathbf{w}%
_{\ell}}\lambda_{K_{2},\mathbf{q}}\right)  . \label{eq:edgeaffine}%
\end{equation}
Since the edge $E$ is critical the facet $F_{i,\mathbf{q}}\subset\partial
K_{i}$, $i=1,2$, opposite of $\mathbf{q}$ lies on one plane and the function%
\[
\phi=%
\begin{cases}
\lambda_{K_{1},\mathbf{q}} & \text{on }K_{1},\\
\lambda_{K_{2},\mathbf{q}} & \text{on }K_{2},
\end{cases}
\]
is globally affine on $K_{1}\cup K_{2}$. As a direct consequence, the
difference in the right in (\ref{eq:edgeaffine}) vanishes and we conclude that
$\int_{\omega_{E}\cap\omega_{E^{\prime}}}p_{k-1}^{E,\mathbf{p}}%
\operatorname{div}B_{E^{\prime}}^{%
\boldsymbol{\mu}%
,\ell}=0$ also holds for $\mathbf{z}=\mathbf{p}$ and $%
\boldsymbol{\mu}%
=\left(  k-2\right)  \mathbf{e}_{\mathbf{z}}^{E^{\prime}}$.

\textbf{Case (ii): }Let ${\ \omega_{E}=\cup_{i=1}^{4}}${$K$}${_{i}}$ with the
local enumeration as in Notation \ref{NotE}. We choose $\mathbf{q}%
\in\mathcal{V}\left(  E\right)  $ so that $E=\left[  \mathbf{q},\mathbf{p}%
\right]  $. Then we have%
\[
\int_{\omega_{\mathbf{z}}}p_{k-1}^{E,\mathbf{p}}\operatorname{div}B_{E}^{%
\boldsymbol{\mu}%
,\ell}=\sum_{i=1}^{4}\frac{\left(  -1\right)  ^{i}}{\left\vert K_{i}%
\right\vert }\int_{K_{i}}P_{k-1}^{\left(  0,3\right)  }\left(  1-2\lambda
_{K_{i},\mathbf{p}}\right)  \partial_{\mathbf{w}_{\ell}}
\left(\boldsymbol{\lambda}_{K_{i},E}^{\boldsymbol{\mu} 		
+\mathbf{1}_{E}}\right).
\]
In view of (\ref{eq:intTi}) this integral can be different from zero only if $%
\boldsymbol{\mu}%
=\left(  k-2\right)  \mathbf{e}_{\mathbf{p}}^{E}$. For $%
\boldsymbol{\mu}%
=\left(  k-2\right)  \mathbf{e}_{\mathbf{p}}^{E}$ and $\widetilde{%
\boldsymbol{\mu}%
}:=\left(  k-1\right)  \mathbf{e}_{\mathbf{p}}^{E}+\mathbf{e}_{\mathbf{q}}%
^{E}$, we conclude from (\ref{eq:intTi}) that%
\begin{align}
\int_{\omega_{\mathbf{z}}}p_{k-1}^{E,\mathbf{p}}\operatorname{div}B_{E}^{%
\boldsymbol{\mu}%
,\ell}  &  =\sum_{i=1}^{4}\frac{\left(  -1\right)  ^{i}}{\left\vert
K_{i}\right\vert }\int_{K_{i}}P_{k-1}^{\left(  0,3\right)  }(1-2\lambda
_{K_{i},\mathbf{p}})\partial_{\mathbf{w}_{\ell}}\left(
 \boldsymbol{\lambda}%
_{K_{i},E}^{\widetilde{%
\boldsymbol{\mu}%
}}\right) \nonumber\\
&  =\sum_{i=1}^{4}\frac{\left(  -1\right)  ^{i}}{\left\vert K_{i}\right\vert
}\int_{K_{i}}P_{k-1}^{\left(  0,3\right)  }(1-2\lambda_{K_{i},\mathbf{p}%
})\lambda_{K_{i},\mathbf{p}}^{k-1}\partial_{\mathbf{w}_{\ell}}\lambda
_{K_{i},\mathbf{q}}\nonumber\\
&  =\frac{1}{\left\vert \hat{K}\right\vert }\int_{\hat{K}}P_{k-1}^{\left(
0,3\right)  }(1-2\lambda_{\mathbf{\hat{p}}})\lambda_{\mathbf{\hat{p}}}%
^{k-1}\sum_{i=1}^{4}\left(  -1\right)  ^{i}\partial_{\mathbf{w}_{\ell}}%
\lambda_{K_{i},\mathbf{q}}. \label{eq:sum4}%
\end{align}
We choose $\mathbf{w}_{1}\in\mathbb{R}^{3}$ as a unit 
vector tangential to the
critical edge $E$. By continuity, we have that $\partial_{\mathbf{w}_{1}%
}\lambda_{K_{i},\mathbf{q}}=\partial_{\mathbf{w}_{1}}\lambda_{K_{j}%
,\mathbf{q}}$ for $1\leq i,j\leq4$. The integral vanishes due to the
alternating signs in the sum. Let $\mathbf{w}_{2}$ be a unit vector
perpendicular to $\mathbf{w}_{1}$ and such that it lies in the plane through
$F_{4}$ (and $F_{2}$). Then again by continuity, we have that
\[
\partial_{\mathbf{w}_{2}}\lambda_{K_{1},\mathbf{q}}=\partial_{\mathbf{w}_{2}%
}\lambda_{K_{4},\mathbf{q}},
\]
and
\[
\partial_{\mathbf{w}_{2}}\lambda_{K_{2},\mathbf{q}}=\partial_{\mathbf{w}_{2}%
}\lambda_{K_{3},\mathbf{q}},
\]
so the sum in the right-hand side of (\ref{eq:sum4}) is zero. A similar
argument can be used if we choose $\mathbf{w}_{3}$ to be a unit vector
perpendicular to $\mathbf{w}_{1}$ which lies in the plane through $F_{1}$ (and
$F_{3}$).

\textbf{Basis functions of Type (d):} In this case, the basis functions are 
given
by
\[
\left.  {B_{\mathbf{z}}^{\ell}}\right\vert _{K}=\mathbf{w}_{\ell}%
\lambda_{K,\mathbf{z}}^{k},\qquad\forall K\in\mathcal{T}_{\mathbf{z}},
\]
for linearly independent vectors $\mathbf{w}_{\ell}\in\mathbb{R}^{3}$,
$1\leq\ell\leq3$, to be fixed below. We distinguish the following two relevant
cases; for all other cases the integral is zero due to $\left\vert \omega
_{E}\cap\omega_{\mathbf{z}}\right\vert =0$).

\begin{enumerate}
\item[(i)] $\mathbf{z}\in\mathcal{V}_{\Omega}$ is an endpoint of $E$: then the
common support is the union of four tetrahedra

\item[(ii)] $\mathbf{z}\in\mathcal{V}_{\Omega}$ is not an endpoint of $E$ but
$\mathbf{z}\in\omega_{E}$: then the common support is the union of two tetrahedra
\end{enumerate}

We will consider both cases by introducing the number $\iota_{\mathbf{z}}$ of
tetrahedra in the common support. We use Notation \ref{NotE} and assume
w.l.o.g. that $K_{1}\cup K_{2}\subset\omega_{E}\cap\omega_{\mathbf{z}}$.
Therefore
\begin{align}
\int_{\Omega}p_{k-1}^{E,\mathbf{p}}\operatorname{div}B_{\mathbf{z}}^{\ell}  &
=k\sum_{i=1}^{\iota_{\mathbf{z}}}\frac{\left(  -1\right)  ^{i}}{\left\vert
K_{i}\right\vert }\left(  \int_{K_{i}}P_{k-1}^{\left(  0,3\right)
}(1-2\lambda_{K_{i},\mathbf{p}})\lambda_{K_{i},\mathbf{z}}^{k-1}\right)
\partial_{\mathbf{w}_{\ell}}\lambda_{K_{i},\mathbf{z}}%
\nonumber\label{eq:int8patch}\\
&  =\left(  \frac{k}{|\hat{K}|}\int_{\hat{K}}P_{k-1}^{\left(  0,3\right)
}\left(  1-2\lambda_{\mathbf{\hat{p}}}\right)  \lambda_{\mathbf{\hat{z}}%
}^{k-1}\right)  \sum_{i=1}^{\iota_{\mathbf{z}}}\left(  -1\right)  ^{i}%
\partial_{\mathbf{w}_{\ell}}\lambda_{K_{i},\mathbf{z}}.
\end{align}
From Lemma \ref{LemOrthoMain}, we conclude that this is zero for $\mathbf{z}%
\neq\mathbf{p}$. For $\mathbf{p}=\mathbf{z}$, which implies $\iota
_{\mathbf{z}}=4$, let $\mathbf{w}_{1}\in\mathbb{R}^{3}$ be the vector
tangential to the critical edge $E$ and let $\mathbf{w}_{2},\ 
\mathbf{w}_{3}%
\in\mathbb{R}^{3}$ be two unit vectors perpendicular to $\mathbf{w}_{1}$ and
such that they lie on the two planes, respectively. By continuity of
$B_{\mathbf{z}}^{\ell}$ the terms in the sum cancel in all cases due to
changing signs.

The following proposition summarizes these findings.

\begin{proposition}
\label{prop:innercritical} Let $k\geq1$. For any critical edge $E\in
\mathcal{E}_{\Omega}$ and any $\mathbf{p}\in\mathcal{V}\left(  E\right)  $,
let $p_{k-1}^{E,\mathbf{p}}\in\mathbb{P}_{k-1,0}\left(  \mathcal{T}\right)  $
be as in Definition \ref{PressCritic}. Then
\[
\left(  p_{k-1}^{E,\mathbf{p}},\operatorname{div}\mathbf{v}\right)
_{L^{2}\left(  \Omega\right)  }=0,\qquad\forall\mathbf{v}\in\mathbf{S}%
_{k,0}\left(  \mathcal{T}\right).
\]
\end{proposition}

\subsection{Critical pressure for outer critical
edges\label{sect:criticalpressure-outer}}

In this section, we consider critical edges that lie on the boundary of the
domain $\Omega$ and construct corresponding critical pressures.

\begin{definition}
[Critical pressure for outer edges]\label{DefCritPout}Let $k\geq1$.
Let $E\in\mathcal{E}_{\partial\Omega}$ be an outer critical edge in
$\mathcal{T}$ and let $\mathcal{T}_{E}=\left\{  K_{i}:1\leq i\leq\iota
_{E}\right\}  $, for some $\iota_{E}\in\{1,~2,~3\}$. For $\mathbf{p\in}%
\mathcal{V}\left(  E\right)  $, the \emph{critical pressure} $p_{k-1}%
^{E,\mathbf{p}}\in\mathbb{P}_{k-1,0}\left(  \mathcal{T}\right)  $ is given by%
\[
p_{k-1}^{E,\mathbf{p}}:=\sum_{i=1}^{\iota_{E}}\frac{\left(  -1\right)  ^{i}%
}{\left\vert K_{i}\right\vert }\chi_{K_{i}}P_{k-1}^{\left(  0,3\right)
}\left(  1-2\lambda_{K_{i},\mathbf{p}}\right)  ,
\]
where, again, $\chi_{K_{i}}$ denotes the characteristic function on $K_{i}$.
\end{definition}

In the following, we prove that for any outer critical edge \(E\in 
\mathcal{E}_{\partial\Omega}\) and any \(p\in \mathcal{V}(E)\)
\[
\left(  p_{k-1}^{E,\mathbf{p}}, \operatorname{div}\mathbf{v}\right)
_{L^{2}\left(  \Omega\right)  }=0\quad\forall\mathbf{v}\in\mathbf{S}%
_{k,0}\left(  \mathcal{T}\right)  .
\]
Similar as in the previous section, we evaluate the integral over $\omega_{E}$
and consider the four types (\ref{casea})-(\ref{cased}) of the basis functions 
of 
$\mathbf{S}_{k,0}\left(
\mathcal{T}\right)  $ listed in Remark \ref{RemTypes}.
As before, the cases (\ref{casea}) and (\ref{caseb}) are straightforward and we 
omit to present
this computation.

\textbf{Case (c):} Since edges in $\mathcal{E}_{\partial\Omega}$
do not carry degrees of freedom for $\mathbf{S}_{k,0}\left(  \mathcal{T}%
\right)  $ we can restrict to inner edges 
$E^{\prime}\in\mathcal{E}_{\Omega}$.
In particular this implies $E\neq E^{\prime}$. We have to consider two
non-trivial subcases

\begin{enumerate}
\item[(c.i)] $\mathcal{T}_{E}$ and $\mathcal{T}_{E^{\prime}}$ share two 
tetrahedra,

\item[(c.ii)] $\mathcal{T}_{E}$ and $\mathcal{T}_{E^{\prime}}$ share one 
tetrahedron.
\end{enumerate}

\textbf{Case (c.i):} W.l.o.g. the edge $E^{\prime}$ is shared by $K_{1}$ 
and
$K_{2}$ and we set $\mathbf{q}=E\cap E^{\prime}$ for $\mathbf{q}\in
\mathcal{V}\left(  E\right)  $, i.e. $E^{\prime}=\left[  \mathbf{z}%
,\mathbf{q}\right]  $ for some $\mathbf{z}\in\mathcal{V}$. The 
basis functions
for the velocity for this edge are given by
\[
{B_{E^{\prime}}^{%
\boldsymbol{\mu}%
,\ell}}=\mathbf{w}_{\ell}%
\begin{cases}%
 \boldsymbol{\lambda}%
_{K,E^{\prime}}^{%
\boldsymbol{\mu}%
+\mathbf{1}_{E^{\prime}}} & \text{on }K\in\mathcal{T}_{E^{\prime}},\\
0 & \text{otherwise,}%
\end{cases}
\]
where $\mu\in\mathbb{N}_{0}^{2},\ |%
\boldsymbol{\mu}%
|=k-2$. Next, we compute%
\[
\left(  p_{k-1}^{E,\mathbf{p}},\operatorname{div}{B_{E^{\prime}}^{%
\boldsymbol{\mu}%
,\ell}}\right)  _{L^{2}\left(  \Omega\right)  }=\int_{\omega_{E}\cap
\omega_{E^{\prime}}}p_{k-1}^{E,\mathbf{p}}\partial_{\mathbf{w}_{\ell}%
}{\boldsymbol{\lambda}
_{K,E^{\prime}}^{%
	\boldsymbol{\mu}
	+\mathbf{1}_{E^{\prime}}}}.
\]
From the same analysis as in (\ref{eq:intTi}), we conclude that this integral
vanishes unless $\mathbf{p}=\mathbf{q}$ and $%
\boldsymbol{\mu}%
=\left(  k-2\right)  \mathbf{e}_{\mathbf{p}}^{E^{\prime}}$. In this case, we
have%
\begin{align}
\int_{\omega_{E}\cap\omega_{E^{\prime}}}p_{k-1}^{E,\mathbf{p}}%
\operatorname{div}B_{E^{\prime}}^{%
\boldsymbol{\mu}%
,\ell}  &  =\sum_{i=1}^{2}\frac{\left(  -1\right)  ^{i}}{\left\vert
K_{i}\right\vert }\left(  \int_{K_{i}}P_{k-1}^{\left(  0,3\right)  }\left(
1-2\lambda_{K_{i},\mathbf{p}}\right)  \lambda_{K_{i},\mathbf{p}}^{k-1}\right)
\partial_{\mathbf{w}_{\ell}}\lambda_{K_{i},\mathbf{z}}\nonumber\\
&  =\frac{1}{|\hat{K}|}\int_{\hat{K}}P_{k-1}^{\left(  0,3\right)  }\left(
1-2\lambda_{\mathbf{\hat{p}}}\right)  \lambda_{\mathbf{\hat{p}}}^{k-1}%
\sum_{i=1}^{2}\left(  -1\right)  ^{i}\partial_{\mathbf{w}_{\ell}}%
\lambda_{K_{i},\mathbf{z}}. \label{eq:int2tetraout}%
\end{align}
Since the facets $F_{i}\subset\partial K_{i}$ opposite to $\mathbf{z}$,
$i=1,2$, lie in one plane, the function $\varphi:K_{1}\cup K_{2}%
\rightarrow\mathbb{R}$, $\left.  \varphi\right\vert _{K_{i}}:=\lambda
_{K_{i},\mathbf{z}}$, $i=1,2$ is affine on $K_{1}\cup K_{2}$ and the sum on
the right-hand side in (\ref{eq:int2tetraout}) vanishes due to the alternating sign.

\textbf{Case (c.ii):} Let $K=\omega_{E}\cap\omega_{E^{\prime}}$. In this 
case
$E,E^{\prime}$ are edges of $K$ with empty intersection. 
By repeating the arguments in (\ref{eq:intTi}), it follows that the
integral $\int_{K}p_{k-1}^{E,\mathbf{p}}\operatorname{div}\left(
B_{E^{\prime}}^{\boldsymbol{\mu},\ell}\right)$ vanishes.

\textbf{Case (d):} Let $\mathbf{z}\in\mathcal{V}_{\Omega}$, in particular
$\mathbf{z}\notin\mathcal{V}\left(  E\right)  $. In this case the basis
function is given by
\[
\left.  {B_{\mathbf{z}}^{\ell}}\right\vert _{K}=\mathbf{w}_{\ell}%
\lambda_{K,\mathbf{z}}^{k}%
\]
for $K\in\mathcal{T}_{\mathbf{z}}$. It follows that the integral $\int%
_{\omega_{\mathbf{z}}}p_{k-1}^{E,\mathbf{p}}\operatorname{div}B_{\mathbf{z}%
}^{\ell}$ is zero by repeating the arguments in (\ref{eq:intTi}).

The following proposition summarizes the findings of Section 5.
\begin{proposition}
\label{LemCritPress}Let $k\geq1$ and let $E\in\mathcal{E}$ be a critical edge
in $\mathcal{T}$. For $\mathbf{p}\in\mathcal{V}\left(  E\right)  $, the
critical pressures $p_{k-1}^{E,\mathbf{p}}$ as in Def. \ref{PressCritic}, Def.
\ref{DefCritPout} satisfy
\[
\left(  p_{k-1}^{E,\mathbf{p}},\operatorname{div}\mathbf{v}\right)
_{L^{2}\left(  \Omega\right)  }=0,\qquad\forall\mathbf{v}\in\mathbf{S}%
_{k,0}\left(  \mathcal{T}\right)  .
\]
Consequently, if $\mathcal{T}$ contains a critical edge 
$E\in
\mathcal{E}_{\Omega}$, then $p_{k-1}^{E,\mathbf{p}}$, $\mathbf{p}%
\in\mathcal{V}\left(  E\right)  $, are critical pressures and the pair 
$\left(
\mathbf{S}_{k,0}\left(  \mathcal{T}\right)  ,\mathbb{P}_{k-1,0}\left(
\mathcal{T}\right)  \right)  $ is not inf-sup stable.
\end{proposition}

\section{CR stabilization for critical edges}
\label{SecCRstab}

In this section, we consider critical edges $E\in\mathcal{E}$ which are 
contained in
the nodal patch $\omega_{\mathbf{z}}$ of some $\mathbf{z}\in\mathcal{V}%
_{\Omega}$. We show that the associated critical pressures $p_{k-1}%
^{E,\mathbf{p}}$, $\mathbf{p}\in\mathcal{V}\left(  E\right)  $, are eliminated
by testing $b_{h}\left(  p_{k-1}^{E,\mathbf{p}},~\cdot ~\right)  $ with some
non-conforming Crouzeix-Raviart functions which are locally supported in
$\omega_{\mathbf{z}}$. We distinguish between odd and even polynomial degree.

\subsection{Stabilization for even polynomial degree}

In the following, we prove for even $k\geq 4$ that those 
critical 
pressures
for the conforming $\left(  \mathbf{S}_{k,0}\left(  \mathcal{T}\right)
,\mathbb{P}_{k-1,0}\left(  \mathcal{T}\right)  \right)  $ Stokes element which
have been defined in the previous section are \textquotedblleft
eliminated\textquotedblright\ by basic Crouzeix-Raviart elements.
\begin{theorem}
\label{ThmEven}Let $k\geq4$ be even. Let \(E\in \mathcal{E}\) be a critical 
edge and assume that there is a tetrahedron \(K \in \mathcal{T}_E\), which has 
\(E\) as an edge. Let $\mathbf{w}
_{\ell}\in\mathbb{R}^{3}$, $\ell=1,2,3$, denote three linearly independent
vectors. For $\mathbf{p}\in\mathcal{V}\left(  E\right)  $, consider the 
critical pressure function $p_{k-1}^{E,\mathbf{p}}$ defined in Definition
\ref{PressCritic} (if \(E\in \mathcal{E}_\Omega\)) or Definition 
\ref{DefCritPout} (if \(E\in \mathcal{E}_{\partial\Omega}\)).
Then any function
$p\in\operatorname*{span}\left\{  p_{k-1}^{E,\mathbf{p}}:\mathbf{p}%
\in\mathcal{V}\left(  E\right)  \right\}  $ which satisfies
\[
\left(  p,\operatorname*{div}\mathbf{v}\right)  _{L^{2}\left( \Omega\right)  
	}=0,\qquad\forall\mathbf{v}\in\operatorname*{span}\left\{
B_{k}^{\operatorname*{CR},K}\mathbf{w}_{\ell}:1\leq\ell\leq3\right\}
,\]
is the zero function.
\end{theorem}

\begin{proof}
We choose $K=K_{1}$, where $K_{1}$ is as in Definition \ref{PressCritic} and
Definition \ref{DefCritPout}, respectively. Then,%
\begin{align*}
  \int_{\Omega}p_{k-1}^{E,\mathbf{p}}\operatorname{div}&\left(
B_{k}^{\operatorname*{CR},K}\mathbf{w}_{\ell}\right)  =-\frac{1}{\left\vert
K\right\vert }\int_{K}P_{k-1}^{\left(  0,3\right)  }\left(  1-2\lambda
_{K,\mathbf{p}}\right)  \operatorname{div}\left( \! \left(  \sum_{\mathbf{y}%
\in\mathcal{V}\left(  K\right)  }\!Q_{k}\!\left(  1-2\lambda_{K,\mathbf{y}%
}\right)  -1\right)\!  \mathbf{w}_{\ell}\right) \\
& =\frac{-1}{\left(  k+1\right)  \left\vert K\right\vert }\int%
_{K}P_{k-1}^{\left(  0,3\right)  }(1-2\lambda_{K,\mathbf{p}}%
)\operatorname{div}\left(  \left(  \sum_{\mathbf{y}\in\mathcal{V}\left(
K\right)  }\left(  L_{k+1}-L_{k}\right)  ^{\prime}\left(  1-2\lambda
_{K,\mathbf{y}}\right)  \right)  \mathbf{w}_{\ell}\right) \\
& =\frac{2}{\left(  k+1\right)  \left\vert K\right\vert }\int%
_{K}P_{k-1}^{\left(  0,3\right)  }\left(  1-2\lambda_{K,\mathbf{p}}\right)
\left(  \sum_{\mathbf{y}\in\mathcal{V}\left(  K\right)  }\left(  L_{k+1}%
-L_{k}\right)  ^{\prime\prime}\left(  1-2{\lambda}_{K,\mathbf{y}}\right)
\partial_{\mathbf{w}_{\ell}}\lambda_{K,\mathbf{y}}\right) \\
& =\frac{2}{\left(  k+1\right)  \left\vert K\right\vert }%
\sum_{\mathbf{y}\in\mathcal{V}\left(  K\right)  }\partial_{\mathbf{w}_{\ell}%
}\lambda_{K,\mathbf{y}}\int_{K}P_{k-1}^{\left(  0,3\right)  }\left(
1-2\lambda_{K,\mathbf{p}}\right)  \left(  \left(  L_{k+1}-L_{k}\right)
^{\prime\prime}\left(  1-2{\lambda}_{K,\mathbf{y}}\right)  \right)  .
\end{align*}
We employ an affine pullback $\chi_{K,\mathbf{y}}:\hat{K}\rightarrow K$ to the
reference element which depends on the summation index $\mathbf{y}%
\in\mathcal{V}\left(  K\right)  $ such that $\chi_{K,\mathbf{y}}^{-1}\left(
\mathbf{p}\right)  =\mathbf{\hat{p}}=\left(  1,0,0\right)  ^{T}$ and therefore
$\lambda_{K,\mathbf{p}}\circ\chi_{K,\mathbf{y}}\left(  \mathbf{x}\right)
=x_{1}$. For $\mathbf{y}\in\mathcal{V}\left(  K\right)  \backslash\left\{
\mathbf{p}\right\}  $, we require in addition that $\chi_{K,\mathbf{y}}$
satisfies $\chi_{K,\mathbf{y}}^{-1}\left(  \mathbf{y}\right)  =\mathbf{\hat
{y}}=\left(  0,1,0\right)  ^{T}$ and $\lambda_{K,\mathbf{y}}\circ
\chi_{K,\mathbf{y}}\left(  \mathbf{x}\right)  =x_{2}$. Then,
\[
\int_{\Omega}p_{k-1}^{E,\mathbf{p}}\operatorname{div}\left(
B_{k}^{\operatorname*{CR},K}\mathbf{w}_{\ell}\right)  =\frac{2}{k+1}\frac
{1}{\left\vert \hat{K}\right\vert }\sum_{\mathbf{y}\in\mathcal{V}\left(
K\right)  
}\partial_{\mathbf{w}_{\ell}}\lambda_{K,\mathbf{y}} 
I_{\mathbf{y}},
\]
with%
\begin{equation}
I_{\mathbf{y}}:=%
\begin{cases}
\int_{\hat{K}}P_{k-1}^{(0,3)}(1-2x_{1})\left(  L_{k+1}-L_{k}\right)
^{\prime\prime}\left(  1-2x_{1}\right)  & \qquad\mathbf{y}=\mathbf{p},\\
\int_{\hat{K}}P_{k-1}^{\left(  0,3\right)  }\left(  1-2x_{1}\right)  \left(
L_{k+1}-L_{k}\right)  ^{\prime\prime}\left(  1-2{x}_{2}\right)  &
\qquad\mathbf{y}\in\mathcal{V}\left(  K\right)  \backslash\left\{
\mathbf{p}\right\}  .
\end{cases}
\label{DefIy}%
\end{equation}
These integrals are computed in Appendix 
\ref{AppJacobi}. From there, we get
\begin{equation}
I_{\mathbf{y}}=\left\{
\begin{array}
[c]{ll}%
\frac{k+1}{4} & \mathbf{y}=\mathbf{p},\\
\frac{1}{2}\left(  -1\right)  ^{k-1} & \mathbf{y}\in\mathcal{V}\left(
K\right)  \backslash\left\{  \mathbf{p}\right\}  .
\end{array}
\right.  \label{Iyvalue}%
\end{equation}
Hence, by taking into account that $k$ is even, we have%
\[
\int_{\Omega}p_{k-1}^{E,\mathbf{p}}\operatorname{div}\left(
B_{k}^{\operatorname*{CR},K}\mathbf{w}_{\ell}\right)  =\frac{1}{k+1}\frac
{1}{\left\vert \hat{K}\right\vert }\left(  \frac{k+1}{2}\partial
_{\mathbf{w}_{\ell}}\lambda_{K,\mathbf{p}}-\sum_{\mathbf{y}\in\mathcal{V}%
\left(  K\right)  \backslash\left\{  \mathbf{p}\right\}  }\partial
_{\mathbf{w}_{\ell}}\lambda_{K,\mathbf{y}}\right)  .
\]
We use $\sum_{\mathbf{y}\in\mathcal{V}\left(  K\right)  }\lambda
_{k,\mathbf{y}}=1$, and obtain%
\[
\int_{\Omega}p_{k-1}^{E,\mathbf{p}}\operatorname{div}\left(
B_{k}^{\operatorname*{CR},K}\mathbf{w}_{\ell}\right)  =\frac{k+3}{k+1}\frac
{1}{2\left\vert \hat{K}\right\vert }\partial_{\mathbf{w}_{\ell}}%
\lambda_{K,\mathbf{p}}.
\]
The assertion is proved if there exist two vectors $\mathbf{s},\mathbf{t}%
\in\mathbb{R}^{3}$ such that the Gram's matrix \(\mathbf{m}=\left(
\partial_{\mathbf{r}}\lambda_{K,\mathbf{p}}\right)  _{\substack{\mathbf{p}%
\in\mathcal{V}\left(  E\right)  \\\mathbf{r}\in\left\{  \mathbf{s}%
,\mathbf{t}\right\}  }}\in\mathbb{R}^{2\times2}\) is regular. Let
$\mathbf{v}\in\mathcal{V}\left(  K\right)  \backslash\mathcal{V}\left(
E\right)  $ and $E=\left[  \mathbf{p},\mathbf{q}\right]  $. We choose
$\mathbf{s}=\mathbf{q}-\mathbf{p}$ and $\mathbf{t}=\mathbf{q}-\mathbf{v}$ to
obtain $\partial_{\mathbf{s}}\lambda_{K,\mathbf{p}}=-1$, $\partial
_{\mathbf{s}}\lambda_{K,\mathbf{q}}=1$, $\partial_{\mathbf{t}}\lambda
_{K,\mathbf{p}}=0$, $\partial_{\mathbf{t}}\lambda_{K,\mathbf{q}}=1$. This
implies that $\mathbf{m}$ is regular and the theorem is proved.%
\end{proof}

\subsection{Stabilization for odd polynomial degree}

In this section, we prove an analogue to Theorem \ref{ThmEven} for odd
$k$. However, we need more Crouzeix-Raviart functions in order to eliminate
the critical pressures as can be seen from the following theorem.

\begin{theorem}
Let $k\geq3$ be odd. Let \(E\in \mathcal{E}\) be a critical edge and assume 
that there is a tetrahedron
\(K\in \mathcal{T}_E\) such that there are two facets \(F,G\in 
\mathcal{F}(K)\) which satisfy 
 \[F,G\in 
 \mathcal{F}_\Omega, \qquad E\subset F,\qquad E\not\subset G.\]
 For
$\mathbf{p}\in\mathcal{V}\left(  E\right)  $, consider the critical pressure 
functions $p_{k-1}^{E,\mathbf{p}}$ defined in Definition \ref{PressCritic}  (if 
\(E\in 
\mathcal{E}_{\Omega}\)) or
Definition \ref{DefCritPout} (if \(E\in\mathcal{E}_{\partial\Omega}\)). For any 
$F^\prime\in\mathcal{F}_{\Omega}$, let
$\mathbf{w}_{\ell}^{F^\prime}\in\mathbb{R}^{3}$, $1\leq\ell\leq3$, denote some 
basis
in $\mathbb{R}^{3}$. Then, any function $p\in\operatorname*{span}\left\{
p_{k-1}^{E,\mathbf{p}}:\mathbf{p}\in\mathcal{V}\left(  E\right)  \right\}  $
which satisfies
\[
\left(  p,\operatorname*{div}\mathbf{v}\right)  _{L^{2}\left(  \Omega\right)  
}=0\quad\forall\mathbf{v}\in\operatorname*{span}\left\{
B_{k}^{\operatorname*{CR},F^\prime}\mathbf{w}_{\ell}^{F^\prime}:1\leq\ell\leq 3,
F^\prime\in\{F,G\}\right\}
\]
is the zero function.
\end{theorem}

\begin{proof}
We will use the two facets \(F, G\) with corresponding
Crouzeix-Raviart functions as test functions and derive the assertion. First,
let $F\in\mathcal{F}_{\Omega}$ be such that $E\subset F$. We recall,
that $B_{k}^{\operatorname*{CR},F}$ has support on two tetrahedra $K_{1}%
,K_{2}$ and $K_{1}\cup K_{2}= \omega_{E}\cap \omega_{F}$. For
$\mathbf{s}\in\mathbb{R}^{3}\backslash\left\{  0\right\}  $, we compute%
\[
\left(  p_{k-1}^{E,\mathbf{p}},\operatorname*{div}\left(  B_{k}%
^{\operatorname*{CR},F}\mathbf{s}\right)  \right)  _{L^{2}\left(
\Omega\right)  }=\sum_{i=1}^{2}\frac{\left(  -1\right)  ^{i}%
}{\left\vert K_{i}\right\vert }\int_{K_{i}}P_{k-1}^{\left(  0,3\right)
}\left(  1-2\lambda_{K_{i},\mathbf{p}}\right)  \operatorname{div}\left(
Q_{k}(1-2\lambda_{K_{i},\mathbf{v}_{i}})\mathbf{s}\right)  ,
\]
where $\mathbf{v}_{i}$ is the vertex in $\mathcal{V}\left(  K_{i}\right)  $
opposite to $F$. For a single summand we get%
\begin{align*}
\int_{K_{i}}  &  P_{k-1}^{\left(  0,3\right)  }(1-2\lambda_{K_{i},\mathbf{p}%
})\operatorname{div}\left(  Q_{k}\left(  1-2\lambda_{K_{i},\mathbf{v}_{i}%
}\right)  \mathbf{s}\right) \\
&  =-2\left(  \partial_{\mathbf{s}}\lambda_{K_{i},\mathbf{v}_{i}}\right)
\int_{K_{i}}P_{k-1}^{\left(  0,3\right)  }\left(  1-2\lambda_{K_{i}%
,\mathbf{p}}\right)  Q_{k}^{\prime}\left(  1-2\lambda_{K_{i},\mathbf{v}_{i}%
}\right) \\
&  =-\frac{2}{k+1}\left(  \partial_{\mathbf{s}}\lambda_{K_{i},\mathbf{v}_{i}%
}\right)  \frac{|K_{i}|}{|\hat{K}|}\int_{\hat{K}}P_{k-1}^{\left(  0,3\right)
}\left(  1-2\lambda_{\mathbf{\hat{p}}}\right)  \left(  L_{k+1}-L_{k}\right)
^{\prime\prime}\left(  1-2\lambda_{\mathbf{\hat{v}}}\right)  ,
\end{align*}
where we used an affine pullback $\chi_{K_{i}}:\hat{K}\rightarrow K_{i}$ which
satisfies $\chi_{K_{i}}^{-1}\left(  \mathbf{p}\right)  =\mathbf{\hat{p}%
}=\left(  1,0,0\right)  ^{T}$ and $\chi_{K_{i}}^{-1}\left(  \mathbf{v}%
_{i}\right)  =\mathbf{\hat{v}}=\left(  0,1,0\right)  ^{T}$. Recall the
computations (\ref{DefIy}), (\ref{Iyvalue}) from the previous section. Then,%
\[
\int_{\hat{K}}P_{k-1}^{\left(  0,3\right)  }\left(  1-2\lambda_{\mathbf{\hat
{p}}}\right)  \left(  L_{k+1}-L_{k}\right)  ^{\prime\prime}\left(
1-2\lambda_{\mathbf{\hat{v}}}\right)  =\frac{1}{2}\left(  -1\right)  ^{k-1}.
\]
The combination of these computations leads to%
\begin{align*}
\left(  p_{k-1}^{E,\mathbf{p}},\operatorname*{div}B_{k}^{\operatorname*{CR}%
,F}\mathbf{s}\right)  _{L^{2}\left(  \omega_{\mathbf{z}}\right)  }  &
=\sum_{i=1}^{2}\frac{\left(  -1\right)  ^{i}}{\left\vert K_{i}\right\vert
}\int_{K_{i}}P_{k-1}^{\left(  0,3\right)  }\left(  1-2\lambda_{K_{i}%
,\mathbf{p}}\right)  \operatorname{div}\left(  Q_{k}(1-2\lambda_{K_{i}%
,\mathbf{v}_{i}})\mathbf{s}\right) \\
&  =%
\frac{\left(  -1\right)  ^{k-1}}{\left(  k+1\right)  \left\vert \hat
{K}\right\vert }\partial_{\mathbf{s}}\left(  \lambda_{K_{1},\mathbf{v}_{1}%
}-\lambda_{K_{2},\mathbf{v}_{2}}\right)  .
\end{align*}%
Since the right-hand side does not depend on $\mathbf{p}$, the test functions
$B_{k}^{\operatorname*{CR},F}\mathbf{s}$, $\mathbf{s}\in\mathbb{R}^{3}$, are
not sufficient to eliminate both critical functions $p_{k-1}^{E,\mathbf{p}}$,
$\mathbf{p}\in\mathcal{V}\left(  E\right)  $.

Next, we choose a facet for another Crouzeix-Raviart function. Let 
$K:=K_{1}$ and $G\in\mathcal{F}_{\Omega}\backslash
\left\{  F\right\}  $ be an inner facet which satisfies \(E\not\subset G\). 
This 
implies
\[\mathcal{T}_{G}\cap\mathcal{T}_{E}=\left\{
K\right\}  .
\]
Let $\mathbf{y}$ denote the vertex in $K$ opposite to $G$ and hence
$\mathbf{y}\in\mathcal{V}\left(  E\right)  $. Next, we compute%
\[
\left(  p_{k-1}^{E,\mathbf{p}},\operatorname*{div}B_{k}^{\operatorname*{CR}%
,G}\mathbf{t}\right)  _{L^{2}\left( \Omega\right)  }=-\frac
{1}{\left\vert K\right\vert }\int_{K}P_{k-1}^{\left(  0,3\right)  }\left(
1-2\lambda_{K,\mathbf{p}}\right)  \operatorname{div}\left(  Q_{k}\left(
1-2\lambda_{K,\mathbf{y}}\right)  \mathbf{t}\right)  .
\]
For $\mathbf{p}=\mathbf{y}$, we employ an affine transform $\chi_{K}:\hat
{K}\rightarrow K$ with $\chi_{K}^{-1}\left(  \mathbf{p}\right)  =\mathbf{\hat
{p}}=\left(  1,0,0\right)  ^{T}$. Then,
\[
\left(  p_{k-1}^{E,\mathbf{p}},\operatorname*{div}B_{k}^{\operatorname*{CR}%
,G}\mathbf{t}\right)  _{L^{2}\left( \Omega\right)  }=\frac
{2}{\left\vert \hat{K}\right\vert }\left(  \partial_{\mathbf{t}}%
\lambda_{K,\mathbf{y}}\right)  \int_{\hat{K}}P_{k-1}^{\left(  0,3\right)
}\left(  1-2\lambda_{\mathbf{\hat{p}}}\right)  Q_{k}^{\prime}\left(
1-2\lambda_{\mathbf{\hat{p}}}\right)  \overset{\text{(\ref{Iyvalue})}}{=}%
\frac{1}{2\left\vert \hat{K}\right\vert }\left(  \partial_{\mathbf{t}}%
\lambda_{K,\mathbf{y}}\right)  .
\]%
For $\mathbf{p}\neq\mathbf{y}$, we employ an affine transform $\chi_{K}:\hat
{K}\rightarrow K$ with $\chi_{K}^{-1}\left(  \mathbf{p}\right)  =\mathbf{\hat
{p}}=\left(  1,0,0\right)  ^{T}$ and $\chi_{K}^{-1}\left(  \mathbf{y}\right)
=\mathbf{\hat{y}}=\left(  0,1,0\right)  ^{T}$. Then,
\begin{align*}
\left(  p_{k-1}^{E,\mathbf{p}},\operatorname*{div}B_{k}^{\operatorname*{CR}%
,G}\mathbf{t}\right)  _{L^{2}\left( \Omega\right)  }&=\frac
{2}{\left\vert \hat{K}\right\vert }\left(  \partial_{\mathbf{t}}%
\lambda_{K,\mathbf{y}}\right)  \int_{\hat{K}}P_{k-1}^{\left(  0,3\right)
}\left(  1-2\lambda_{\mathbf{\hat{p}}}\right)  Q_{k}^{\prime}\left(
1-2\lambda_{\mathbf{\hat{y}}}\right)  \\
&\overset{\text{(\ref{Iyvalue})}}{=}%
\frac{1}{\left(  k+1\right)  \left\vert \hat{K}\right\vert }\left(
\partial_{\mathbf{t}}\lambda_{K,\mathbf{y}}\right)  .
\end{align*}%
We define the Gram's matrix%
\[
\mathbf{m}= 
\begin{bmatrix}
\partial_{\mathbf{s}}\left(  \lambda_{K_{1},\mathbf{v}_{1}}-\lambda
_{K_{2},\mathbf{v}_{2}}\right)  &\partial_{\mathbf{t}}%
\lambda_{K,\mathbf{y}}\\
\partial_{\mathbf{s}}\left(  \lambda_{K_{1},\mathbf{v}_{1}}-\lambda
_{K_{2},\mathbf{v}_{2}}\right)  & \frac{k+1}{2}\partial_{\mathbf{t}%
}\lambda_{K,\mathbf{y}}
\end{bmatrix},
\]
and choose $\mathbf{s}$ as the unit vector which is orthogonal to the facet
$F$ and points into $K_{2}$. In this way, $\partial_{\mathbf{s}}\left(
\lambda_{K_{1},\mathbf{v}_{1}}-\lambda_{K_{2},\mathbf{v}_{2}}\right)
=:\theta<0$. We choose $\mathbf{t}:=\mathbf{y}-\mathbf{u}$ for some
$\mathbf{u}\in\mathcal{V}\left(  K\right)  \backslash\left\{  \mathbf{y}%
\right\}  $. Hence, $\partial_{\mathbf{t}}\lambda_{K,\mathbf{y}}=1$ and%
\[
\mathbf{m}=\begin{bmatrix}
\theta & 1\\
\theta & \frac{k+1}{2}%
\end{bmatrix}
\]
Since $k\geq3$ this matrix is non-singular and this implies the claim.%
\end{proof}

\appendix

\section{Non-conforming Crouzeix-Raviart functions in higher
	dimensions}
\label{AppHigherD}

The construction of non-conforming Crouzeix-Raviart functions in three
dimensions is based on the definition of the univariate polynomial $Q_{k}%
\in\mathbb{P}_{k}\left(  \left[  -1,1\right]  \right)  $ in (\ref{defQk2}). In
this section, we generalize this construction to arbitrary dimension
$d\geq2$ and $Q_{k}$ in (\ref{defQk2}) will be a special case for $d=3$.

Let $K\subset\mathbb{R}^{d}$ be a closed simplex with vertices $\mathbf{z}%
_{i}$, $1\leq i\leq d+1$, which form the set $\mathcal{V}\left(  K\right)  $.
The $\left(  d-1\right)  $-dimensional facet in $\partial K$ opposite to
$\mathbf{z}\in\mathcal{V}\left(  K\right)  $ is denoted by $F_{\mathbf{z}}$
and the facets are collected in the set $\mathcal{F}\left(  K\right)  $. The
barycentric coordinates $\lambda_{K,\mathbf{z}}$, $\mathbf{z}\in
\mathcal{V}\left(  K\right)  $, are characterized by the conditions
$\lambda_{K,\mathbf{z}}\in\mathbb{P}_{1}\left(  K\right)  $, $\lambda
_{K,\mathbf{z}}\left(  \mathbf{y}\right)  =\delta_{\mathbf{z},\mathbf{y}}$ for
all $\mathbf{y,z}\in\mathcal{V}\left(  K\right)  $.

Our goal is to define a polynomial $Q_{d,k}\in\mathbb{P}_{k}\left(  \left[
-1,1\right]  \right)  $, $d\geq2$, $k\geq1$, such that the composition
$Q_{d,k}\left(  1-2\lambda_{K,\mathbf{z}}\right)  $ satisfies%
\begin{equation}
\left.  Q_{d,k}\left(  1-2\lambda_{K,\mathbf{z}}\right)  \right\vert
_{F_{\mathbf{z}}}=1\text{ and }\forall F\in\mathcal{F}\left(  K\right)
\backslash\left\{  F_{\mathbf{z}}\right\}  :\;\int_{F}Q_{d,k}\left(
1-2\lambda_{K,\mathbf{z}}\right)  q=0\quad\forall q\in\mathbb{P}_{k-1}\left(
F\right)  .\label{PropQdk}%
\end{equation}
Following the construction in Section \ref{SecCanCR}, the non-conforming
Crouzeix-Raviart functions for even polynomial degree $k\geq2$ are supported
on a single simplex $K$ and given by%
\[
B_{d,k}^{\operatorname*{CR},K}:=\left\{
\begin{array}
[c]{ll}%
\left(  \sum_{\mathbf{z}\in\mathcal{V}\left(  K\right)  }Q_{d,k}\left(
1-2\lambda_{K,\mathbf{z}}\right)  \right)  -1 & \text{on }K\text{,}\\
0 & \text{otherwise.}%
\end{array}
\right.
\]
For odd polynomial degree $k\geq1$ they are supported on the two adjacent
simplices $K_{1}$, $K_{2}$ of an inner facet $F$ and given by%
\begin{equation}
B_{d,k}^{\operatorname*{CR},F}:=\left\{
\begin{array}
[c]{ll}%
Q_{d,k}\left(  1-2\lambda_{K,\mathbf{z}}\right)   & \text{for }K\in\left\{
K_{1},K_{2}\right\}  ,\\
0 & \text{otherwise,}%
\end{array}
\right.
\end{equation}
where $\lambda_{K,\mathbf{z}}$ denotes the barycentric coordinate for the
vertex $\mathbf{z}\in\mathcal{V}\left(  K\right)  $ opposite to $F$.

Properties (\ref{PropQdk}) allow us to repeat the arguments in the proof of
Theorem \ref{thm:CR-tetrahedron} which then imply that $B_{d,k}%
^{\operatorname*{CR},K}$ and $B_{d,k}^{\operatorname*{CR},F}$ belong to the
Crouzeix-Raviart space $\operatorname*{CR}\nolimits_{k,0}^{\max}\left(
\mathcal{T}\right)  $ for a conforming simplicial finite element mesh
$\mathcal{T}$ of a $d$-dimensional polytope $\Omega$.

Finally, we construct the polynomial $Q_{d,k}$. For ease of notation, we set
$m=d-2$ and define the polynomial $P_{k+m}\in\mathbb{P}_{k+m}\left(  \left[
-1,1\right]  \right)  $ as a linear combination of Legendre polynomials%
\[
P_{k+m}=\sum_{\ell=0}^{m}\beta_{k,\ell}L_{k+\ell}.
\]
The coefficients $\beta_{k,\ell}$ are defined as the solutions of the linear
system%
	\begin{equation}
\frac{1}{2^{n}n!}\sum_{\ell=0}^{m}\beta_{k,\ell}\frac{\left(  
	\ell+k+n\right)
	!}{\left(  \ell+k-n\right)  !}=\delta_{n,m},\qquad0\leq n\leq 
m,\label{Condbetakl}%
\end{equation}
where we use the convention that $\mu!/\nu!=0$ for $\mu\geq0$ and $\nu<0$.
\begin{lemma}
	Let $d\geq2$ and set $m=d-2$. The polynomial
	\[
	Q_{d,k}=P_{k+m}^{\left(  m\right)  }%
	\]
	belongs to $\mathbb{P}_{k}\left(  \left[  -1,1\right]  \right)  $ and
	satisfies the conditions in (\ref{PropQdk}).
\end{lemma}%

\begin{proof}
Let $C_{n}^{\left(  \alpha\right)  }\in\mathbb{P}_{n}\left(  \left[
-1,1\right]  \right)  $ denote the Gegenbauer polynomials for $\alpha>-1/2$,
$\alpha\neq0$. We combine \cite[18.7.9]{NIST:DLMF} with \cite[18.9.19]%
{NIST:DLMF} to get for $\ell\leq n$%

\begin{equation}
L_{n}^{\left(  \ell\right)  }=\frac{\left(  2\ell\right)  !}{2^{\ell}\ell
	!}C_{n-\ell}^{\left(  1/2+\ell\right)  }.\label{Gegenbauer}%
\end{equation}
We use \cite[Table 18.6.1]{NIST:DLMF} to obtain%
\[
L_{n}^{\left(  \ell\right)  }\left(  1\right)  =\frac{\left(  2\ell\right)
	!}{2^{\ell}\ell!}\frac{\left(  1+2\ell\right)  _{n-\ell}}{\left(
	n-\ell\right)  !}=\frac{\left(  n+\ell\right)  !}{2^{\ell}\ell!\left(
	n-\ell\right)  !}.
\]
Thus,%
\[
Q_{d,k}\left(  1\right)  =P_{k+m}^{\left(  m\right)  }\left(  1\right)
=\sum_{\ell=0}^{m}\beta_{k,\ell}L_{k+\ell}^{\left(  m\right)  }\left(
1\right)  =\sum_{\ell=0}^{m}\beta_{k,\ell}\frac{\left(  k+\ell+m\right)
	!}{2^{m}m!\left(  k+\ell-m\right)  !}.
\]
Condition (\ref{Condbetakl}) for $n=m$ shows $Q_{d,k}\left(  1\right)  =1$.
Since $Q_{d,k}\left(  \left.  \left(  1-2\lambda_{K,\mathbf{z}}\right)
\right\vert _{F_{\mathbf{z}}}\right)  =Q_{d,k}\left(  1\right)  $ the first
condition in (\ref{PropQdk}) follows.

Next, we prove the second condition in (\ref{PropQdk}). For $\mathbf{z}%
\in\mathcal{V}\left(  K\right)  $, let $F\in\mathcal{F}\left(  K\right)
\backslash\left\{  F_{\mathbf{z}}\right\}  $. We employ an affine pullback
$\chi_{K}:\hat{K}\rightarrow K$ to the reference element
\[
\hat{K}:=\left\{  \mathbf{x}=\left(  x_{i}\right)  _{i=1}^{d}\in
\mathbb{R}_{\geq0}^{d}\mid x_{1}+\ldots+x_{d}\leq1\right\}
\]
in such a way that $\hat{F}:=\left\{  \mathbf{x}=\left(  x_{i}\right)
_{i=1}^{d}\in\hat{K}\mid x_{d}=0\right\}  $ is mapped to $F$. Then, it is
sufficient to prove%
\begin{equation}
\int_{\hat{F}}Q_{d,k}\left(  1-2x_{1}\right)  x_{1}^{\alpha_{1}}\ldots
x_{d-1}^{\alpha_{d-1}}dx_{d-1}\ldots dx_{1}=0,\qquad\forall%
{\boldsymbol\alpha}%
=\left(  \alpha_{i}\right)  _{i=1}^{d-1}\in\mathbb{N}_{0}^{d-1},\quad
\left\vert
\boldsymbol\alpha
\right\vert \leq k-1.\label{CondQdk}%
\end{equation}
We set $%
\boldsymbol\alpha
^{\prime}:=\left(  \alpha_{i}\right)  _{i=2}^{d-1}$ and define%
\[
G\left(  x_{1}\right)  :=\int_{0}^{1-x_{1}}\int_{0}^{1-x_{1}-x_{2}}\ldots
\int_{0}^{1-x_{1}-x_{2}-\ldots-x_{d-2}}x_{1}^{\alpha_{1}}\ldots x_{d-1}%
^{\alpha_{d-1}}dx_{d-1}\ldots dx_{2}.
\]
Hence, condition (\ref{CondQdk}) is equivalent to%
\[
\int_{0}^{1}Q_{d,k}\left(  1-2x_{1}\right)  G\left(  x_{1}\right)
dx_{1}=0.
\]
By using integration by parts, one derives that this condition is equivalent to%
\begin{align}
0=\int_{0}^{1}Q_{d,k}\left(  1-2x_{1}\right)  G\left(
x_{1}\right)  dx_{1}  & =-\sum_{\ell=1}^{m}\left.  2^{-\ell}P_{k+m}^{\left(
	m-\ell\right)  }\left(  1-2x_{1}\right)  G^{\left(  \ell-1\right)  }\left(
x_{1}\right)  \right\vert _{0}^{1}\label{Zerocond1}\\
& +2^{-m}\int_{0}^{1}P_{k+m}\left(  1-2x_{1}\right)  G^{\left(  m\right)
}\left(  x_{1}\right)  dx_{1}.\nonumber
\end{align}
The integral in the definition of $G$ can be evaluated explicitly and we get%
\[
G\left(  x_{1}\right)  =\frac{%
\boldsymbol\alpha
	^{\prime}!}{\left(  m+\left\vert
\boldsymbol\alpha
	^{\prime}\right\vert \right)  !}x_{1}^{\alpha_{1}}\left(  1-x_{1}\right)
^{m+\left\vert
\boldsymbol\alpha
	^{\prime}\right\vert }.
\]
Clearly, $G\in\mathbb{P}_{m+k-1}\left(  \left[  0,1\right]  \right)  $ and
$G^{\left(  \ell-1\right)  }\left(  1\right)  =0$ for all $1\leq\ell\leq m$.
We use this, $G^{\left(  m\right)  }\in\mathbb{P}_{k-1}\left(  \left[
0,1\right]  \right)  $, and the orthogonality relations of the Legendre
polynomials in (\ref{Zerocond1}) to get the equivalent condition%
\[
0=\sum_{\ell=1}^{m}2^{-\ell}P_{k+m}^{\left(  
m-\ell\right)
}\left(  1\right)  G^{\left(  \ell-1\right)  }\left(  0\right)  .
\]
However, the property $P_{k+m}^{\left(  m-\ell\right)  }\left(  1\right)  =0$
for $1\leq\ell\leq m$ follows from the first conditions $n=0,1,\ldots,m-1$ in
the definition of $\beta_{k,\ell}$ in (\ref{Condbetakl}).%
\end{proof}

\begin{remark}
	The $m$-th order derivative of $P_{k+m}$ in the definition of $Q_{d,k}$ can 
	be
	avoided by employing the relation (\ref{Gegenbauer}) for Gegenbauer and
	Legendre polynomials. We get with $m=d-2$%
	\[
	Q_{d,k}=\frac{\left(  2m\right)  !}{2^{m}m!}\sum_{\ell=0}^{m}\beta_{k,\ell
	}C_{k+\ell-m}^{\left(  1/2+m\right)  }%
	\]
	where $C_{\nu}^{\left(  \lambda\right)  }$ is set to zero for $\nu<0$ and we
	emphasize that the superscript $\left(  1/2+m\right)  $ does \textbf{not}
	denote a derivative but is the parameter in the Gegenbauer polynomial 
	related
	to the corresponding weight function $\left(  1-x^{2}\right)  ^{m}$ in the
	orthogonality relation. The coefficients $\beta_{k,\ell}$ can be expressed
	explicitly%
	\[
	\beta_{k,\ell}=\frac{\left(  -1\right)  ^{m-\ell}2^{m}\binom{m}{\ell}\left(
		2k+2\ell+1\right)  }{
		{\textstyle\prod\nolimits_{r=\ell+1}^{m+\ell+1}}
		\left(  2k+r\right)  }%
	\]
	(the verification that $\beta_{k,\ell}$ satisfy (\ref{Condbetakl}) is quite
	tedious and skipped) so that we obtain the fully explicit formula%
	\begin{align*}
	Q_{d,k}=\frac{\left(  2m\right)  
	!}{m!}\nsum[1.5]_{\ell=0}^{m}\left({
	{\displaystyle\prod\limits_{\substack{r=\ell+1\\r\neq2\ell+1}}^{m+\ell+1}}
	\left(  2k+r\right)^{-1}  }\right) {\left(
		-1\right)  ^{m-\ell}\binom{m}{\ell}}C_{k+\ell-m}^{\left(  1/2+m\right)  
		}.
	\end{align*}
	
\end{remark}

\section{Computing some integrals involving Jacobi
polynomials}
\label{AppJacobi}

In this appendix, we evaluate the integrals $I_{\mathbf{y}}$ defined in
(\ref{DefIy}).

\begin{lemma}
For any $\mathbf{y}\in\mathcal{V}\left(  K\right)  $, the integral
$I_{\mathbf{y}}$ in (\ref{DefIy}) is explicitly given by%
\[
I_{\mathbf{y}}=\left\{
\begin{array}
[c]{ll}%
\frac{k+1}{4} & \mathbf{y}=\mathbf{p},\\
\frac{1}{2}\left(  -1\right)  ^{k-1} & \mathbf{y}\in\mathcal{V}\left(
K\right)  \backslash\left\{  \mathbf{p}\right\}  .
\end{array}
\right.
\]

\end{lemma}

\begin{proof}
We start with $\mathbf{y}=\mathbf{p}$. We evaluate the inner integral
explicitly and apply integration by parts to get%
\begin{align*}
I_{\mathbf{p}}  &  =\int_{0}^{1}P_{k-1}^{\left(  0,3\right)  }\left(
1-2x_{1}\right)  \left(  L_{k+1}-L_{k}\right)  ^{\prime\prime}\left(
1-2x_{1}\right)  \left(  \int_{0}^{1-x_{1}}\int_{0}^{1-x_{1}-x_{2}}%
1dx_{3}dx_{2}\right)  dx_{1}\\
&  =\frac{1}{2}\int_{0}^{1}\left(  x_{1}-1\right)  ^{2}P_{k-1}^{\left(
0,3\right)  }\left(  1-2x_{1}\right)  \left(  L_{k+1}-L_{k}\right)
^{\prime\prime}\left(  1-2x_{1}\right)  dx_{1}\\
&  =\left.  -\frac{1}{4}\left(  x_{1}-1\right)  ^{2}P_{k-1}^{\left(
0,3\right)  }\left(  1-2x_{1}\right)  \left(  L_{k+1}-L_{k}\right)  ^{\prime
}\left(  1-2x_{1}\right)  \right\vert _{0}^{1}\\
&  -\frac{1}{4}\int_{0}^{1}\left(  \left(  x_{1}-1\right)  ^{2}P_{k-1}%
^{\left(  0,3\right)  }\left(  1-2x_{1}\right)  \right)  ^{\prime}\left(
L_{k+1}-L_{k}\right)  ^{\prime}\left(  1-2x_{1}\right)  dx_{1}\\
&  =\frac{1}{4}P_{k-1}^{\left(  0,3\right)  }\left(  1\right)  \left(
L_{k+1}-L_{k}\right)  ^{\prime}\left(  1\right) \\
&  +\frac{1}{8}\left.  \left(  \left(  x_{1}-1\right)  ^{2}P_{k-1}^{\left(
0,3\right)  }\left(  1-2x_{1}\right)  \right)  ^{\prime}\left(  L_{k+1}%
-L_{k}\right)  \left(  1-2x_{1}\right)  \right\vert _{0}^{1}\\
&  -\frac{1}{8}\int_{0}^{1}\underset{g\left(  x_{1}\right)
}{\underbrace{\left(  \left(  x_{1}-1\right)  ^{2}P_{k-1}^{\left(  0,3\right)
}\left(  1-2x_{1}\right)  \right)  ^{\prime\prime}}}\left(  L_{k+1}%
-L_{k}\right)  \left(  1-2x_{1}\right)  dx_{1}.
\end{align*}
Since $g\in\mathbb{P}_{k-1}$, the orthogonality properties of the Legendre
polynomials imply that the last term vanishes. Hence,%
\begin{align*}
I_{\mathbf{p}}  &  =\frac{1}{4}P_{k-1}^{\left(  0,3\right)  }\left(  1\right)
\left(  L_{k+1}-L_{k}\right)  ^{\prime}\left(  1\right) \\
&  -\frac{1}{8}\left.  \left(  \left(  x_{1}-1\right)  ^{2}P_{k-1}^{\left(
0,3\right)  }\left(  1-2x_{1}\right)  \right)  ^{\prime}\left(  L_{k+1}%
-L_{k}\right)  \left(  1-2x_{1}\right)  \right\vert _{x_{1}=0}.
\end{align*}
The endpoint properties of the Legendre and Jacobi polynomials (cf.
(\ref{Lendpoints}), (\ref{Pnormalization})) imply that the second term
vanishes and%
\[
I_{\mathbf{p}}=\frac{k+1}{4}.
\]
Next, we consider the integral for $\mathbf{y}\neq\mathbf{p}$. We get again by
integration by parts
\begin{align*}
I_{\mathbf{y}}  &  =\int_{0}^{1}P_{k-1}^{\left(  0,3\right)  }\left(
1-2x_{1}\right)  \left(  \int_{0}^{1-x_{1}}\left(  1-x_{1}-x_{2}\right)
\left(  L_{k+1}-L_{k}\right)  ^{\prime\prime}\left(  1-2{x}_{2}\right)
dx_{2}\right)  dx_{1}\\
&  =-\frac{1}{2}\int_{0}^{1}P_{k-1}^{\left(  0,3\right)  }\left(
1-2x_{1}\right)  \left.  \left(  \left(  1-x_{1}-x_{2}\right)  \left(
L_{k+1}-L_{k}\right)  ^{\prime}\left(  1-2{x}_{2}\right)  \right)  \right\vert
_{0}^{1-x_{1}}dx_{1}\\
&  -\frac{1}{2}\int_{0}^{1}P_{k-1}^{\left(  0,3\right)  }\left(
1-2x_{1}\right)  \int_{0}^{1-x_{1}}\left(  L_{k+1}-L_{k}\right)  ^{\prime
}\left(  1-2{x}_{2}\right)  dx_{2}dx_{1}\\
&  =\frac{1}{2}\left(  L_{k+1}-L_{k}\right)  ^{\prime}\left(  1\right)
\int_{0}^{1}P_{k-1}^{\left(  0,3\right)  }\left(  1-2x_{1}\right)  \left(
1-x_{1}\right)  dx_{1}\\
&  +\frac{1}{4}\int_{0}^{1}P_{k-1}^{\left(  0,3\right)  }\left(
1-2x_{1}\right)  \left.  \left(  L_{k+1}-L_{k}\right)  \left(  1-2{x}%
_{2}\right)  \right\vert _{0}^{1-x_{1}}dx_{1}\\
&  =\frac{1}{2}\left(  L_{k+1}-L_{k}\right)  ^{\prime}\left(  1\right)
\int_{0}^{1}P_{k-1}^{\left(  0,3\right)  }\left(  1-2x_{1}\right)  \left(
1-x_{1}\right)  dx_{1}\\
&  +\frac{1}{4}\int_{0}^{1}P_{k-1}^{\left(  0,3\right)  }\left(
1-2x_{1}\right)  \left(  L_{k+1}-L_{k}\right)  \left(  2x_{1}-1\right)
dx_{1},
\end{align*}
where we used $\left(  L_{k+1}-L_{k}\right)  \left(  1\right)  =0$ for the
last equality. Again by the orthogonality properties of the Legendre
polynomials, the last summand is zero and we get%
\begin{equation}
I_{\mathbf{y}}=\frac{k+1}{8}\iota_{k-1}\quad\text{with }\iota_{k}:=\int%
_{-1}^{1}P_{k}^{\left(  0,3\right)  }\left(  t\right)  \left(  t+1\right)  dt.
\label{defIytemp}%
\end{equation}
We employ \cite[18.9.5]{NIST:DLMF} for $\beta=2$, $\alpha=0$, $n=k$, i.e.,%
\[
\left(  2k+3\right)  P_{k}^{(0,2)}=\left(  k+3\right)  P_{k}^{\left(
0,3\right)  }+kP_{k-1}^{\left(  0,3\right)  },
\]
to obtain%
\[
\iota_{k}=-\frac{k}{k+3}\iota_{k-1}+\frac{2k+3}{k+3}\int_{-1}^{1}P_{k}%
^{(0,2)}\left(  t\right)  \left(  t+1\right)  dt.
\]
The last integral has been computed in \cite[Lem. C.1]{CCSS_CR_1}, and we
obtain%
\[
\iota_{k}=-\frac{k}{k+3}\iota_{k-1}+4\left(  -1\right)  ^{k}\frac{\left(
2k+3\right)  }{\left(  k+1\right)  \left(  k+2\right)  \left(  k+3\right)  }.
\]
For $k=0$, it holds $P_{0}^{\left(  0,3\right)  }\left(  x\right)  =1$ and
$\iota_{0}=2$. It is easy to verify by induction that $\iota_{k}:=4\left(
-1\right)  ^{k}/\left(  k+2\right)  $ satisfies the initial value and the
recurrence. The combination with (\ref{defIytemp}) leads to the assertion.%
\end{proof}

\def\cprime{$'$}

\end{document}